\documentclass[12pt]{amsart}

\usepackage{mathptmx, amscd, amssymb, amsmath, enumerate, slashbox, 
float}

\usepackage{amsthm}
\usepackage{amsmath}
\usepackage{amssymb}
\usepackage[ansinew]{inputenc}
\usepackage{amscd}
\usepackage[ansinew]{inputenc}
\usepackage[mathscr]{euscript}
\usepackage{color, lipsum}
\usepackage[mathscr]{euscript}
\usepackage{alltt}

\makeatletter
\newcommand{\verbatimfont}[1]{\def\verbatim@font{#1}}%
\makeatother

\numberwithin{equation}{section}

\newtheorem{thm}{Theorem}[section]
\newtheorem{cor}[thm]{Corollary}
\newtheorem{lem}[thm]{Lemma}
\newtheorem{prop}[thm]{Proposition}

\theoremstyle{definition}
\newtheorem{ex}[thm]{Example}

\theoremstyle{definition}
\newtheorem{defn}[thm]{Definition}

\theoremstyle{definition}

\theoremstyle{definition}
\newtheorem{question}[thm]{Question}

\theoremstyle{definition}
\newtheorem{rem}[thm]{Remark}

\theoremstyle{definition}

\def\C{\mathbb C}

\def\cZ{\mathcal Z}
\def\SS{\mathcal S}
\def\Z{\mathbb Z}
\def\N{\mathbb N}

\def\dim{\operatorname{dim}}

\def\supp{\operatorname{supp}}

\def\rank {\mathrm{rank}}

\def\O{\mathcal O}

\def\syz{\operatorname{syz}}

\def\m{\mathbf m}

\def\FF{\mathcal F}
\def\f{\overline f}

\def\TT{\mathcal T}

\def\I{\mathbf I}
\def\tL{\textnormal{\texttt L}}
\def\w{\mathbf w}

\def\SS{\mathcal S}
\def\R{\mathcal R}

\def\k{\mathbf k}

\newcommand{\Sym}{\operatorname{Sym}}
\newcommand{\Spec}{\operatorname{Spec}}
\newcommand{\Proj}{\operatorname{Proj}}
\newcommand{\Hom}{\operatorname{Hom}}
\newcommand{\fm}{\mathfrak{m}}
\newcommand{\fn}{\mathbf{n}}
\newcommand{\fp}{\mathfrak{p}}
\newcommand{\NN}{\mathbb{N}}
\newcommand{\ZZ}{\mathbb{Z}}
\newcommand{\Ass}{\operatorname{Ass}}
\newcommand{\Min}{\operatorname{Min}}

\def\f0{\mathbf{0}}

\def\geq{\geqslant}
\def\leq{\leqslant}
\def\gs{\geqslant}
\def\ls{\leqslant}

\def\*{{\color{red}\blacksquare}}
\def\+{\color{green}\blacksquare}

\subjclass[$2000$ Mathematics Subject Classification]{Primary 13B22; Secondary 13H15,  32S05}

\begin{document}

\title[Analytic spread and integral closure of integrally decomposable modules]{Analytic spread and integral closure of integrally \\ \vskip3pt decomposable modules}



\author{Carles Bivi\`a-Ausina}
\address{
Institut Universitari de Matem\`atica Pura i Aplicada,
Universitat Polit\`ecnica de Val\`encia,
Cam\'i de Vera, s/n,
46022 Val\`encia, Spain}
\email{carbivia@mat.upv.es}

\author{Jonathan Monta\~no}
\address{
Department of Mathematical Sciences, New Mexico State University,
1290 Frenger Mall,
MSC 3MB / Science Hall 236,
Las Cruces, 88003-8001 New Mexico, USA }
\email{jmon@nmsu.edu }


\thanks{The first author was partially supported by MICINN Grant PGC2018-094889-B-I00}

\begin{abstract}
We relate the analytic spread of a module expressed as the direct sum of two submodules with the analytic spread of its components.
We also study a class of submodules
whose integral closure can be expressed in terms of the integral closure of its row ideals, and therefore can be obtained by means of a simple computer algebra procedure.
In particular, we analyze a class of modules, not necessarily of maximal rank, whose integral closure is determined by the
family of Newton polyhedra of their row ideals.
\end{abstract}

\maketitle


\section{Introduction}\label{intro}


Given an ideal $I$ in a Noetherian local ring $(R, \m)$, the notions of integral closure, reduction, analytic spread, and
multiplicity of $I$ 
are fundamental objects of study in commutative algebra and algebraic geometry (see for instance \cite{HS, HIO, V}).
These notions have essential applications also in singularity theory mainly due to the works
of Lejeune and Teissier \cite{LT,Cargese, TeissierIM}. These applications concern the study of the equisingularity of deformations of hypersurfaces in $(\C^n,0)$ with isolated singularity at the origin.
The concept of integral closure of ideals was extended by Rees to modules (see \cite{Rees87}).
Moreover, the multiplicity of ideals was extended to modules
by Buchsbaum and Rim \cite{BRim} (see also Kirby \cite{Kirby}), thus leading to what is commonly known as Buchsbaum-Rim multiplicity of a submodule of $R^p$ of finite colength.


The integral closure and multiplicity of a submodule of a free module satisfy analogous properties as those satisfied by ideals.
For instance, they satisfy an analogous of the Rees' multiplicity theorem (see \cite{Katz} or \cite[Corollary 8.20]{V}). Moreover, when the residual field is infinite, the analytic spread of a submodule (see Definition \ref{defanspread}) also coincides with the minimum number of elements needed to generate a reduction of the submodule (see \cite{BUV,HS,V}). We also remark that, by the results of Gaffney \cite{Gaffney92, Gaffney96} the notion of integral closure of modules and Buchsbaum-Rim multiplicities have essential
applications to the study of the equisingularity of deformations of isolated complete intersection singularities. We also refer to \cite{GTW} for other applications in singularity theory.

In general, the computation of the analytic spread and the integral closure of a submodule is a non-trivial problem than can be approached from several points of view. 
Our objective in this work takes part of the general project of computing effectively the
analytic spread and the integral closure for certain classes of modules.
We relate the analytic spread of a module expressed as the direct sum of two submodules with the analytic spread of its components
(see Theorem \ref{mainAS} and Corollary \ref{cormainAS}). 
Moreover, we analyze a class of submodules $M\subseteq R^p$, that we call {\it integrally decomposable},
for which a generating system of $\overline M$
can be obtained by means of an easy computer algebra procedure once the integral closure of each row ideal $M_i$ is known
(Theorem \ref{WDcentral}).


In Section \ref{preliminars} we recall briefly some fundamental facts about the integral closure of modules, analytic spread, reductions, Buchsbaum-Rim
multiplicity of submodules of a free module and Rees algebras that will be used in subsequent sections. In particular, we highlight the connection between the integral closure
of a module $M$ and the integral closure of the ideal generated by the minors of size $\rank(M)$ of $M$
(Theorem \ref{icm} and Corollary \ref{IrMh}).

Section \ref{tasodm} is devoted to the study of the analytic spread of decomposable modules.  The main result of this section is Theorem \ref{mainAS}, where we relate $\ell(M \oplus N)$ with $\ell(M)$ and $\ell(N)$, and we derive a generalization of some results of
\cite{Hyr02} and \cite{Lipman} about the analytic spread of ideals (see also \cite[8.4.4]{HS}).
This result has required the study of
multi-graded Rees algebras and their corresponding multi-projective spectrum (see Subsections \ref{multigrad1} and \ref{multigrad2}).
As a corollary, given ideals $I_1,\dots, I_p$ of $R$, we prove that $\ell(I_1\oplus \cdots \oplus I_p)=\ell(I_1\cdots I_p)+p-1$ (see Corollary \ref{ideals}).

In Section 4 we introduce the class of integrally decomposable modules $M\subseteq R^p$ (Definition \ref{WDdef}) and analyze their relation
with the condition $C(M)=\overline M$ (see Theorem \ref{WDcentral}), where $C(M)$ denotes the submodule of $R^p$ generated by
the elements $h\in \overline{M_1}\oplus \cdots \oplus \overline{M_p}$ such that $\rank(M)=\rank(M,h)$. In general we have that $\overline M\subseteq C(M)$.
If equality holds, then we obtain a substantial
simplification of the computation of $\overline M$, as can be seen in Examples \ref{exidcd} and \ref{exCzero}.

We also extend the notion of Newton non-degenerate submodule of $\O_n^p$ introduced in \cite{BiviaJLMS} to the case where the rank of the module is not $p$.
These modules constitute a wide class of integrally decomposable submodules. We recall that $\O_n$ denotes the local ring
of analytic function germs $(\C^n,0)\to \C$.
As a consequence of our study we show in Example \ref{deKod} an integrally closed and non-decomposable
submodule of $\O_2^2$ (see Definition \ref{decom}) whose ideal of maximal minors can be factorized as the product of two proper integrally closed ideals.

\section{Preliminaries: Rees algebras, analytic spread, and integral closure}\label{preliminars}


Throughout this paper $R$ is a Noetherian ring and all $R$-modules are finitely generated.  An $R$-module $M$ has  a {\it rank} if there exists $e\in \NN$ such that $M_\fp\cong R^e_\fp$ for every $\fp$ associated prime of $R$. Equivalently, if $M\otimes_R Q(R)$ is a free $Q(R)$-module of rank $e$, where $Q(R)$ is the total ring of fractions of $R$. In this case we also say $M$ has rank $e$ ($\rank(M)=e$), and if $e>0$ we say $M$ has {\it positive rank}.  We note that an $R$-ideal $I$ has positive rank if it contains non-zero divisors. If $R$ is an integral domain, then $Q(R)$ is a field and hence every module over an integral domain has a rank.

From now on, whenever $M$ is a submodule of a free module $R^p$, we identify $M$ with a matrix of generators. In this case, we denote by $\I_i(M)$ the ideal of $R$ generated by the
$i\times i$ minors of $M$. If $i>p$, then we set $\I_i(M)=(0)$. We note that the ideals $\I_i(M)$ are independent of the matrix of generators chosen as they agree with the {\it Fitting ideals} of the module $R^p/M$ (see \cite[Section 2.2]{E}).
If $M$ has a rank, the maximum $i$ such that $\I_i(M)\otimes_R Q(R)\neq (0)$ coincides with $\rank(M)$.

If $M\subseteq R^p$ is a submodule, then for any $\tL\subseteq \{1,\dots, p\}$, $\tL\neq\emptyset$, we denote by $M_{\tL}$ the submodule of $R^{\vert \tL\vert}$
obtained by projecting the components of $M$ indexed by $\tL$, where $\vert \tL\vert$ is the cardinal of $\tL$.
In particular, we have $M_{\{i\}} = M_i$  for all $i=1,\dots, p$, where $M_i$ is the ideal of $R$ generated by the elements of the $i$-th row of any matrix of generators of $M$.
The ideals $M_1,\dots, M_p$ are called the {\it row ideals of $M$}. It is immediate to check that these ideals are independent of the chosen matrix of generators of $M$.

\begin{defn}\label{decom}
Let $M$ be a submodule of $R^p$. We say that $M$ is {\it decomposable} when $M=M_1\oplus \cdots \oplus M_p$. 
\end{defn}

\subsection{Rees algebras and the analytic spread}  In this subsection we include the definition and some of the  properties  of Rees algebras of modules. We also define the analytic spread of modules. For more details see \cite{EHU} and \cite{SUV}.

 Henceforth, we denote by $\Sym_R(M)$ the {\it symmetric algebra} of the $R$-module $M$, or simply $\Sym(M)$ when the base ring is clear. We also denote by $\tau_R(M)$ the {\it $R$-torsion} of $M$, i.e., $\tau_R(M)=\{x\in M\mid (0:_R x) \text{ contains non-zero divisors of } R\}$.

\begin{defn}
If $M$ has a rank, the {\it Rees algebra} of $M$ is defined as
$$\R(M):=\Sym(M)/\tau_R(\Sym(M)).$$
\end{defn}
The above definition coincides with the usual one for ideals, i.e., $\R(I)=R[It]=\oplus_{n\in \NN} I^nt^n$, although we note that the latter does not require the rank assumption.

\begin{rem}\label{torsionless}
Assume $M$ has a rank, then the natural map $\R(M)\to \R(M/\tau_R(M))$ is an isomorphism, i.e., $$\Sym(M)/\tau_R(\Sym(M))\cong \Sym(M/\tau_R(M))/\tau_R\big(\Sym(M/\tau_R(M))\big).$$
To see this, we note that since $\Sym(M/\tau_R(M))/\tau_R\big(\Sym(M/\tau_R(M))\big)$ is torsion-free, the kernel of the natural map $\varphi: \Sym(M)\to \Sym(M/\tau_R(M))/\tau_R\big(\Sym(M/\tau_R(M))\big)$ contains $\tau_R(\Sym(M))$. On the other hand, since $M$ has a rank, $M\otimes_R Q(R)$ is free and then $\varphi\otimes_R Q(R)$ is an isomorphism. Thus,  $\ker(\varphi)$ has rank zero which is equivalent to being contained in $\tau_R(\Sym(M))$.

We also note that $M/\tau_R(M)$ is a torsion-free module with a rank, then it is contained in a free $R$-module. The latter implies that when dealing with  the Rees algebra of a module  with a rank, one may always assume it is contained in a free module.
\end{rem}
\begin{rem}\label{subFree1}
Assume $M$ has a rank and $M/\tau_R(M)\subseteq F$ for a free $R$-module $F\cong R^{r}$, then $\R(M)$ is isomorphic to the image of the map $\Sym(M)\xrightarrow{\alpha} \Sym(F)\cong R[t_1,\ldots, t_r]$.
\end{rem}

In the following proposition we recall some facts about the dimension and associated primes of Rees algebras. 
 Following the notation from Remark \ref{subFree1}, let $\TT := R[t_{1},\ldots, t_r]$. For  any $I \in \Spec R$ we denote by $I'$ the $\R$-ideal $I \TT\cap \R(M)$.

\begin{prop}\label{info}
Let $M$ be an $R$-module that has a rank. Then
\begin{enumerate}
\item[$(1)$] $\Min(\R(M)) = \{P'\mid P\in \Min(R)\}$ and $\Ass(\R(M))=\{P'\mid P\in \Ass(R)\}$.
\item[$(2)$] $\dim \R(M) = \dim R+\rank (M)$.
\end{enumerate}
\end{prop}
\begin{proof}
See \cite[Section 15.4]{Matsumura} and \cite[2.2]{SUV}.
\end{proof}

We are now ready to define the analytic spread.
\begin{defn}\label{defanspread}
Assume $(R,\fm,k)$ is local and $M$ is an $R$-module having a rank. The {\it fiber cone} of $M$ is defined as
$\FF(M):=\R(M)\otimes_R k.$
The {\it analytic spread} of $M$ is then $\ell(M):=\dim \FF(M).$
\end{defn}

The following proposition will be needed in several of our arguments.

\begin{prop}[{\cite[2.3]{SUV}}]\label{posAS}
Let $M$ be an $R$-module having a rank. Then $$\rank (M) \ls \ell(M) \ls \dim R + \rank (M) -1 .$$
\end{prop}

\subsection{Integral closure of modules }
In this subsection we include the definition of integral closure of modules and some basic  properties of it. We restrict ourselves to the case of  torsion-free modules with a rank. 
For more details see \cite[Chapter 16]{HS} and \cite[Chapter 8]{V}.

\begin{defn}[{Rees \cite{Rees87}}]\label{theDefn}
Let $R$ be a Noetherian ring and let $M$ be a submodule of $R^p$.
\begin{enumerate}
\item[$(1)$] The element $h\in R^p$ is {\it integral} over $M$ if for every minimal prime $\fp$ of $R$ and every discrete valuation ring (DVR) or field $V$ between $R/\fp$ and $R_\fp/\fp R_\fp$, the image $hV$ of  $h$ in $V^p$ is in the image $MV$  of the composition of $R$-maps
$M\hookrightarrow R^p\to  V^p$  (see \cite[16.4.9]{HS}).
\item[$(2)$] The {\it integral closure} of $M$ in $R^p$ is defined as $\overline{M}:=\{h\in R^p: h \textnormal{ is integral over }M\},$ which is a  submodule of $R^p$.
If $M=\overline{M}$, we say $M$ is {\it integrally closed}. We note that if $M\subseteq R$ is an ideal, then the integral closure of $M$ as a module coincides with that as an ideal (see \cite[6.8.3]{HS}).

\item[$(3)$]  Assume $M$ has a rank. A submodule $U\subseteq M$ having a rank is a {\it reduction} of $M$ if $M\subseteq \overline{U}$. As shown in \cite{Rees87} (see also \cite[16.2.3]{HS}), this is equivalent to $\R(M)$ being integral over the subalgebra generated by the image of $U$. The latter condition is in turn equivalent to $[\R(M)]_{n+1}=U[\R(M)]_n$ for $n\gg 0$, where $U$ is identified with its image in $[\R(M)]_1$.  A reduction is {\it minimal} if it does not  properly contain any other reduction of $M$.
\end{enumerate}
\end{defn}

\begin{rem}\label{sameRank}
Let $M\subset R^p$ be a submodule having a rank, then
\begin{enumerate}
\item[$(1)$] $\overline{M}=[\overline{\R(M)}]_1$, where $\overline{\R(M)}$ is the integral closure of $\R(M)$ in $\Sym(R^p)$ (cf. \cite[5.2.1]{HS}).

\item[$(2)$] If $R$ is local, then for every reduction $U$ of $M$ we have
 $\mu(U)\gs \ell(M)$,
 where $\mu(-)$ denotes the minimal number of generators. Moreover, if $R$ has infinite residue field then every minimal reduction is generated by exactly $\ell(M)$ elements.

\item[$(3)$]  It is clear from the definition that free modules  $R^q\subseteq R^p$ are integrally closed. 
Moreover, if $U\subseteq M$ is a reduction, then $\rank(U)=\rank(M)$ (see \cite[p.\,416]{V}). In particular, $\rank(M)=\rank(\overline{M})$.
 \end{enumerate}
\end{rem}

The integral closure of modules admits several characterizations. The following theorem relates the integral closure of modules with the integral closure of ideals. 
 As far as the authors are aware, this result had not appeared in the literature in this generality (see \cite[1.7]{Gaffney92}, \cite[16.3.2]{HS}, \cite[1.2]{Rees87}, \cite[8.66]{V} for related statements).

\begin{thm}\label{icm}
Let $R$ be a Noetherian ring and  $M\subseteq R^p$ a submodule having a rank. 
Let $h\in R^p$ be such that $M+Rh$ also has a rank and $\rank(M)=\rank(M+Rh)$. Then the following conditions are equivalent.

\begin{enumerate}
\item[$(1)$] $h\in\overline M$.
\item[$(2)$] $\overline {\I_i(M)}=\overline{\I_i(M+R h)}$, for all $i\geq 1$.
\item[$(3)$] $\overline {\I_r(M)}=\overline{\I_r(M+R h)}$, for $r=\rank(M)$.
\end{enumerate}


\end{thm}

For the proof of the theorem we need the following lemma whose proof is essentially the same as \cite[1.6]{Gaffney92}. We include here the details for completeness.

\begin{lem}\label{Cramer}
Let $R$ be a Noetherian integral domain and $M\subseteq R^p$ a submodule. Let $h\in R^p$ be an arbitrary element and set $r=\rank(M+Rh)$, 
then $\I_r(M)h\subseteq \I_r(M+Rh)M$.
\end{lem}

\begin{proof}
If $\rank(M)<r$, then $\I_r(M)=0$ and the conclusion clearly follows. Then we may assume $\rank(M)=r$. We identify $M$ with a matrix of generators and $M+Rh$ with the matrix $[M\vert h]$. Let $M'$ be a $r\times r$ submatrix of $M$ such that $d=\det(M')\neq 0$ and
let $\tL\subseteq \{1,\ldots, p\}$ be the rows of $M$ corresponding to the rows of $M'$. By Cramer's rule there exists $x_1,\ldots, x_r\in \I_r(M'\vert h_\tL)\subseteq \I_r(M\vert h)$ such that
$Mx= dh_\tL$, where $x=[\begin{matrix}
x_1&\ldots&
x_r
\end{matrix}]^T\in R^r.
$
Let $N$ be the $p\times r$ submatrix of $M$ corresponding to the columns of $M'$ and consider the vector $g=dh-Nx$. By construction, we have $g\in M+Rh$ and $g_\tL=0$. Let $i\in \{1,\ldots, p\}\setminus \tL$, then  the $(r+1)\times (r+1)$ minor of $[N\vert g]$ corresponding to the rows $\tL\cup\{i\}$ is $\pm g_id$ and it must vanish since $\rank(M+Rh)=r$. Therefore, $g_id=0$ which implies $g_i=0$. Thus $g=0$ and then $dh=Nx\subseteq  \I_r(M+Rh)M$. Since $M'$ was chosen arbitrarily the proof is complete.
\end{proof}

We are now ready to prove the theorem.

\begin{proof}[Proof of Theorem \ref{icm}]
We begin with (1) $\Rightarrow$ (2).  Let $\fp$  be a minimal prime of $R$ and  $V$ a DVR or a field  between $R/\fp$ and $R_\fp/\fp R_\fp$. Since $hV\in MV$,  for every $i\gs 1$ we have
$$\I_i(M+Rh)V =\I_i(MV+R(hV))\subseteq \I_i(MV)= \I_i(M)V.$$
Thus $\I_i(M+Rh)\subseteq \overline{\I_i(M)}$ and (2) follows.

Since (2) $\Rightarrow$ (3) is clear, it suffices to show (3) $\Rightarrow$ (1).  Let $\fp$  be a minimal prime of $R$ and for a submodule $N\subseteq R^q$ let $N(R/\fp)$ its image  in $(R/\fp)^q$. By assumption we have that $M(R/\fp)$ and $(M+Rh)(R/\fp)$ both have rank $r$. In particular, $\I_r(M)(R/\fp)=\I_r(M(R/\fp))\neq 0$, and likewise $\I_r(M+Rh)(R/\fp)\neq 0$.  Let $V$ a DVR or a field  between $R/\fp$ and $R_\fp/\fp R_\fp$. Then by the assumption and Lemma \ref{Cramer}, applied to $R/\fp$, we have
$$
(\I_r(M)V)hV=(\I_r(M)h)V\subseteq (\I_r(M+Rh)M)V=(\I_r(M+Rh)V)MV=(\I_r(M)V)MV.
$$
Thus $hV\in MV$. We conclude $h\in \overline{M}$, as desired.
\end{proof}

As an immediate consequence of Theorem \ref{icm} we have the following result.

\begin{cor}\label{IrMh}
Let $R$ be a Noetherian ring and let $M\subseteq R^p$ be a submodule having a rank. Let $r=\rank(M)$. Then
$$
\overline M=\left\{ h\in R^p: \rank(M + R h)=r
\hspace{0.4cm}\textnormal{and}\hspace{0.4cm}\I_r(M+R h)\subseteq \overline{\I_r(M)}\right\}.
$$
\end{cor}

Assume $R$ is local of dimension $d$ and let $\lambda(-)$ denote the length function of $R$-modules. If  $\lambda(R^p/M)<\infty$, we say $M$ has {\it finite colength} and in this case the limit  $$e(M)=(d+p-1)!\lim_{n\to \infty}\frac{\lambda([\Sym(R^p)]_n/[\R(M)]_n)}{n^{d+p-1}}$$
is called the {\it Buchsbaum-Rim multiplicity} of $M$.
It is known that if $R$ is Cohen Macaulay and $M$ is generated by $d+p-1$ elements, then $e(M)=\lambda(R^p/M)=\lambda(R/\I_p(M))$
(see for instance \cite[p.\,214]{Gaffney96}).

We recall the following numerical characterization of integral closures due to Rees \cite{Rees61} in the case of ideals and Katz \cite{Katz} for modules. 

\begin{thm}[{\cite[p.\,317]{HS},\cite{Katz}}]\label{numCr}
Let $R$ be a formally equidimensional Noetherian local ring of dimension $d>0$. 
Let $N\subseteq M\subseteq R^p$ be submodules such that $\lambda(R^p/N)<\infty$. Then $\overline M=\overline{N}$ if and only if $e(N)=e(M)$.
\end{thm}

\begin{rem}\label{basics} Let $M\subseteq R^p$ be a submodule.
In general we have
\begin{equation}\label{icfonam}
\overline M\subseteq \overline{M_1\oplus \cdots \oplus M_p}=\overline{M_1}\oplus \cdots \oplus \overline{M_p}.
\end{equation}
However, the first inclusion in \eqref{icfonam} might be strict. For
instance, consider the
submodule of $\O_2^2$ generated by the columns of the matrix
$$
\left[\begin{matrix}
x+y & x^3 & y^3 \\
x & y & x
\end{matrix}\right].
$$
It is clear that $x^3\in \overline{M_1}$, $x\in\overline{M_2}$. Let
$h=[\begin{matrix} x^3 & x \end{matrix}]^T$, we can see that $h\notin \overline M$.
By Theorem \ref{icm} we have that
$$
h\in\overline M \Longleftrightarrow \I_2(M+ \O_2 h)\subseteq \overline {\I_2(M)}
\Longleftrightarrow e(\I_2(M))=e(\I_2(M+\O_2 h)),
$$
where the last equivalence follows from Theorem \ref{numCr}.
However $e(\I_2(M))=8$ and $e(\I_2(M+ \O_2 h))=6$, as can be computed using Singular \cite{Singular}.  Hence $h\notin \overline M$.

Another argument leading to the conclusion that $h\notin \overline M$ is the following.
We have that $e(M)=7$ and $e(M+  \O_2 h)=5$, computed again using Singular. Since these multiplicities are different, it follows that $h\notin M$, by Theorem \ref{numCr}. Moreover, by using Macaulay2 (see Remark \ref{intClRemark}) it is possible to prove that $\overline M$ is generated by the columns of the matrix
$$
\left[\begin{matrix}
x+y & x^3 & y^3&x^3y^2 \\
x & y & x&x+y
\end{matrix}\right].
$$
That is, $\overline{M}=M+\O_2[\begin{matrix}
x^3y^2&
x+y
\end{matrix}]^T.
$

\end{rem}

Given an analytic map $\varphi:(\C^m,0)\to (\C^n,0)$, we denote by $\varphi^*$ the morphism $\O_n\to \O_m$
given by $\varphi^*(h)=h\circ \varphi$, for all $h\in\O_n$. For submodules of $\O_n^p$ we have the following alternative definition of integral closure.

\begin{thm}[Gaffney {\cite[p.\,303]{Gaffney92}}]\label{morphDef} Let $M\subseteq \O_n^p$ be a submodule and let $h\in \O_n^p$. Then $h$ is {\it integral over $M$} if and only if
$\varphi^*(h)\in \O_1\varphi^*(M)$, for any analytic curve $\varphi:(\C,0)\to (\C^n,0)$.
\end{thm}

\begin{ex}\label{morphEx}
It is also possible to check that $h\notin \overline M$ in the example from Remark \ref{basics} by considering the arc $\varphi:(\C,0)\to (\C^2,0)$ given by
$\varphi(t)=(-t+t^3, t)$, for all $t\in \C$. We have that
$\varphi^*(h)=[\begin{matrix}
(-t+t^3)^3 &
-t+t^3
\end{matrix}]^T
$
and that $\varphi^*(M)$ is generated by the columns of the matrix
$$
\left[\begin{matrix}
t^3 & (-t+t^3)^3 & t^3 \\
-t+t^3 & t & -t+t^3
\end{matrix}\right].
$$
We note that   the first and  third columns of the previous matrix coincide.
If $\varphi^*(h)\in \varphi^*(M)$, then we would have
\begin{equation}\label{elsI2}
\I_2 \left[\begin{matrix}
t^3 & (-t+t^3)^3 \\
-t+t^3 & t
\end{matrix}\right]=
\I_2 \left[\begin{matrix}
t^3 & (-t+t^3)^3 & (-t+t^3)^3 \\
-t+t^3 & t & -t+t^3
\end{matrix}\right].
\end{equation}
The ideal on the left of (\ref{elsI2}) is equal to $( t^4)$ and the ideal on the right of (\ref{elsI2})
is equal to $( t^6)$. Hence $\varphi^*(h)\notin \varphi^*(M)$ and by Theorem \ref{morphDef} it follows that
 $h\not\in \overline{M}$.
 \end{ex}

We finish this section with the following relation between integral closures and projections.

\begin{prop}\label{projections}
Let $R$ be a Noetherian ring and  $M\subseteq R^p$ a submodule, then for every non-empty $\tL\subseteq \{1,\dots, p\}$ we have $(\overline{M})_\tL\subseteq \overline{M_\tL}$.
\end{prop}
\begin{proof}
Fix $h\in \overline{M}$. For every  a minimal prime $\fp$ of $R$ and every   DVR or  field $V$ between $R/\fp$ and $R_\fp/\fp R_\fp$ we have
$h_\tL V=(hV)_\tL\in (MV)_\tL=M_\tL V.$
Thus $h_\tL\in \overline{M_\tL}$. The result follows.
\end{proof}

\section{The analytic spread of decomposable modules}\label{tasodm}

In this section we study the analytic spread of decomposable modules and its relation with the analytic spread  of their components. Our main results are Theorem \ref{mainAS} and its corollaries. We begin with some necessary background information.

\subsection{Multi-graded algebras and multi-projective spectrum}\label{multigrad1}    In this subsection we recall several facts about multi-graded algebras and their multi-homogeneous spectrum, we refer the reader to \cite{Hyry99}  for more information. We start by setting up some notation.

Let $p\in \ZZ_{>0}.$ We denote by $\fn$  the vector $(n_1,\ldots, n_p)\in \NN^p$. For convenience we also set $\mathbf{0}=(0,\ldots, 0)$ and  $\mathbf{1}=(1,\ldots, 1)$ where each of these vectors belongs to $\NN^p$.  We  call the sum $n_1+\cdots+n_p$ the {\it total degree} of $\fn$ and denote it by $|\fn|$.

Let $R$ be a Noetherian  ring and $A=\oplus_{\fn\in \NN^p}A_{\fn}$ a Noetherian $\NN^p$-graded algebra  with $A_{\mathbf{0}}=R$ and generated by the elements of total degree one ({\it standard graded}).
We  denote by $A^\Delta$ the {\it diagonal subalgebra} of $A$, i.e., $A^\Delta=\oplus_{n\in \NN}A_{n\mathbf{1}}$. For every $1\ls i\ls p$ we write
$A^{(i)}=\oplus_{n_i=0}A_{\fn}$. We also consider the following $\NN^p$-homogeneous $A$-ideals
$A^{+}_i=\oplus_{n_i>0}A_{\fn}$ for $1\ls i\ls p$ and $  A^{+}=\oplus_{n_1,\ldots, n_p>0}A_{\fn}.$
We write
$\Proj^p A = \{P\in \Spec A\mid P \text{ is $\NN^p$-homogeneous, and }A^+\not\subset P\}.$
The {\it dimension} of $\Proj^p A$ is one minus the maximal length  of an increasing chain of elements of $\Proj^p A$,  $P_0\subsetneq P_1\subsetneq\cdots \subsetneq P_d$. 
The relation between the dimensions of $\Proj^p A$ and $A$ is explained in the following lemma.

\begin{lem}[{\cite[1.2]{Hyry99}}]\label{dims}
Let $\cZ=\Proj^p A$ and assume $\cZ\neq \emptyset$, then
\begin{enumerate}
\item[$(1)$] $\dim \cZ = \max\{\dim A/P\mid P\in \cZ\}-p\ls \dim A -p$.
	\item[$(2)$]  If $\dim A^{(i)}<\dim A$ for every $1\ls i\ls p$, then $\dim \cZ = \dim A-p$.
\end{enumerate}
\end{lem}

It is possible to give $\Proj^p A$ a structure of scheme and to show that it is isomorphic to  $\Proj^1 A^\Delta$
(see \cite[Part II, Exercise 5.11]{Harts} and also \cite[Lemma 3.2]{Hayasaka2} and \cite[Lemma 7.1]{Kirby-Rees2}).
For the reader's convenience, we provide a proof of the following particular result which suffices for our applications.

\begin{prop}\label{bijection}
Let $\iota : A^\Delta\to A$ be the natural inclusion. Then $\iota^*: \Proj^p A \to \Proj^1 A^\Delta$ is a bijection.
\end{prop}
\begin{proof}
Clearly $\Proj^p A = \emptyset$ if and only if $\Proj^1 A^\Delta = \emptyset$ if and only if $A_{n\mathbf{1}}=0$ for $n\gg 0$, then we may assume these two sets are both non-empty.  For every $1\ls i\ls p$, let $\mathfrak{e}_i=(0,\ldots,0,1,0,\ldots, 0)\in \NN^p$ where the 1 is in the $i$th-position. Fix $0\neq f_i\in A_{\mathfrak{e}_i}$  for $1\ls i\ls p$ and let $f=f_1\cdots f_p$. Since  every element in the localization  $A_f$ is a unit times an element of $A^\Delta_f$, one can easily see that $\iota_f^*$ is bijective.

We first show $\iota^*$ is injective. Let $P_1,P_2\in \Proj^p A$ and assume $\iota^*(P_1)=\iota^*(P_2)$. If $f$ is as above and such that $f\not \in P_1$ (thus $f\not \in P_2$), then by assumption $\iota^*_f(P_1A_f)=\iota^*_f(P_2A_f)$. Hence $P_1A_f=P_2A_f$, which implies $P_1=P_2$.

We now show $\iota^*$ is surjective. Let $P\in \Proj^1 A^\Delta$ and $f\not \in P$ as above. Then there exists $Q\in A$ such that $\iota^*_f(QA_f)=PA^\Delta_f$, which implies $\iota^*(Q)=P$, finishing the proof.
\end{proof}

We end this subsection with the following lemma that will be used in the proofs of our main results.

\begin{lem}\label{posFiber}
Let $A=\oplus_{\fn\in \NN^p}A_n$ be a Noetherian standard $\NN^p$-graded algebra and $\fp\in \Proj^{p-1} A^{(p)}$ (if $p=1$,  $\Proj^0 A_0$ is simply $\Spec A_0$). Fix $e\in \NN$, then the following statements are equivalent.
\begin{enumerate}
\item[$(1)$] There exists a chain of elements in  $\Proj^p A$, $P_0\subsetneq P_1\subsetneq \cdots \subsetneq P_{e-1}$ such that $\fp = P_i\cap A^{(p)}$ for every $0\ls i \ls e-1$.
\item[$(2)$] $\dim Q(A^{(p)}/\fp)\otimes_{A^{(p)}} A \gs e$.
\end{enumerate}
\end{lem}
\begin{proof}
Set $W = Q(A^{(p)}/\fp)\otimes_{A^{(p)}} A$. If (1) holds, then $P_0W\subsetneq \cdots \subsetneq P_{e-1}W \subsetneq ( \fp +A^+_p)W=W^+$ is a chain of prime ideals in $W$. Thus, $\dim W\gs e$ and (2) follows.

Conversely, if (2) holds then $\dim (A/\fp A)_{\fp A+A^+_p}\gs e$. Since associated primes of $\NN^p$-graded rings are $\NN^p$-homogeneous (\cite[A.3.1]{HS}), a direct adaptation of \cite[1.5.8(a)]{BH} to $\NN^p$-graded rings shows that there exist  $\NN^p$-homogeneous  $A$-ideals $\fp A\subseteq P_0\subsetneq \cdots\subsetneq P_{e-1}\subsetneq (\fp A+A^+_p)$ whose images in  the ring $(A/\fp A)_{\fp A+A^+_p}$ are all different. Since $\fp = P_i\cap A^{(p)}$ for every  $0\ls i\ls e-1$, the result follows.
\end{proof}

\subsection{Multi-graded Rees algebras}\label{multigrad2}

 In this subsection we describe a standard multi-graded structure for the Rees algebras of direct sums of modules.



\begin{defn}\label{grading}
Let $M_1,\ldots,M_p$ be $R$-modules having a rank.  We define a natural standard $\NN^p$-graded structure on $\R(M_1\oplus \cdots \oplus M_p)$.  
By \cite[A2.2.c]{E} we have $$\Sym(M_1\oplus\cdots \oplus M_p)\cong \bigotimes_{i=1}^p\Sym(M_i),$$ and since each of the algebras $\Sym(M_i)$ has a  standard $\NN$-grading, we can combine  these to an $\NN^p$-grading of $\R(M_1\oplus \cdots \oplus M_p)$ by setting
$[\bigotimes_{i=1}^p\Sym(M_i)]_\fn=\bigotimes_{i=1}^p \Sym(M_i)_{n_i}. $
\end{defn}

\begin{prop}\label{iterate}
Let $M_1,\ldots, M_p$ be $R$-modules having a rank and set $\R'=\R(M_1\oplus\cdots\oplus M_{p-1})$. Then  there is a natural graded $\R'$-isomorphism $$\R(M_1\oplus\cdots \oplus M_p)\cong \R(M_p\otimes_R \R').$$
\end{prop}
\begin{proof}
 We claim that for any $R$-module $M$ with a rank we have $\tau_{\R'}(M\otimes_R \R')$ is equal to the image $T$ of $\tau_R(M)\otimes_R \R'$ in $M\otimes_{R} \R'$. First observe that by Proposition \ref{info}(1), $M\otimes_R \R'$ has  a rank as $\R'$-module and it is equal to $\rank (M)$. Now, consider a short exact sequence $$0\to \tau_R(M)\to M\to F\cong R^r\to 0.$$
By tensoring with $\R'$ it follows that $T$ contains  $\tau_{\R'}(M\otimes_R \R')$. On the other hand, $T$ has rank zero as $\R'$-module (since $\rank (\tau_R(M))=0$), then it must be $\R'$-torsion. The claim follows. We obtain  the following natural maps
\begin{align*}
\SS :=\bigotimes_{i=1}^p\Sym_R(M_i)&\xrightarrow{\text{\cite[A2.2.c]{E}}} \R'\otimes_R \Sym_R(M_p)\\
 &\stackrel{\text{\cite[A2.2.b]{E}}}{\cong}\Sym_{\R'}(M_p\otimes_{R} \R' )\\
&\,\,\,\,\,\,\,\,\,\,\xrightarrow{\text{onto}} \,\,\,\,\,\,\,\, \Sym_{\R'}(M_p\otimes_{R} \R' )/\tau_R( \Sym_{\R'}(M_p\otimes_{R} \R' ))\\
 &\,\,\,\,\,\,\,\,\,\stackrel{\text{claim}}{=} \,\,\,\,\,\,\,\,\R(M_p\otimes_{R} \R').
\end{align*}
Clearly the  kernel of the composition of these maps contains $\tau_R(\SS)$ and, since tensoring by $Q(R)$ leads to an monomorphism, this kernel  must be equal to $\tau_R(\SS)$. The result follows.

\end{proof}


\subsection{Main results about the analytic spread of modules} This subsection contains the main results of this section. We assume $(R,\fm,k)$ is a Noetherian local ring.

 The following is the main theorem of this section. This result, in particular, allows us to recover, and extend, the results in \cite[Lemma 4.7]{Hyr02}, \cite[5.5]{Lipman}, and \cite[2.3]{SUV}. 

\begin{thm}\label{mainAS}
Let  $M$ and $N$ be $R$-modules having a rank. Then
$$ \max\{\ell(M)+\rank(N), \ell(N)+ \rank(M)\} \ls \ell(M\oplus  N)\ls \ell(M)+\ell(N).$$
\end{thm}
\begin{proof}
We may assume $M$ and $N$ are torsion-free and hence contained in free $R$-modules (Remark \ref{torsionless}). If either $M$ or $N$ has rank zero, then it has to be the zero module. Then we may assume  they both have positive rank. Consider the following natural surjective maps
 $$\R(M)\otimes_R \R(N)\xrightarrow{\alpha} \R(M)\otimes_R \R(N)/\tau_R(\R(M)\otimes_R \R(N))\xleftarrow{\beta} \Sym(M)\otimes_R \Sym(N).$$
Since $Q(R)\otimes_R \beta$ is an isomorphism and the image of $\beta$ is torsion-free, it follows that $\ker \beta\subseteq  \tau_R(\Sym(M)\otimes_R \Sym(N))\subseteq \ker \beta .$ Then we obtain a surjective map
\begin{align*}
\R(M)\otimes_R \R(N)&\,\,\,\,\,\,\,\,\,\,\xrightarrow{ \text{onto}}\,\,\,\,\,\,\,\,\,\,\Sym(M)\otimes_R \Sym(N)/\tau_R(\Sym(M)\otimes_R \Sym(N))\\
 & \stackrel{\text{\cite[A2.2.c]{E}}}{\cong} \Sym(M\oplus N)/\tau_R(\Sym(M\oplus N))\\
 &\,\,\,\,\,\,\,\,\,\,\,\,\,=\,\,\,\,\,\,\,\,\,\,\R(M\oplus N).
\end{align*}
By tensoring this map by $k$ we observe that $\FF(M\oplus N)$ is a quotient of $\FF(M)\otimes_k \FF(N)$, and since the latter is a tensor product of affine algebras, it has dimension $\dim \FF(M)+\dim \FF(N)=\ell (M)+\ell(N)$.  The  right-hand inequality follows.

We now show the left-hand inequality. Set $\R = \R(M\oplus N) $. Following the multi-grading in Definition \ref{grading} we have  $\R^{(1)}=\R(M)$. We also observe that $\R \cong \R(N')$,  where  $N'=N\otimes_R \R^{(1)}$ (Proposition \ref{iterate}).
Fix $\fp\in \Proj^1 \R^{(1)}$ such that  $\fp\cap R=\fm$ and $\dim \R^{(1)}/\fp= \ell(M)$, which exists by Lemma \ref{posFiber} and the fact that $\ell(M)\gs 1$ (Proposition \ref{posAS}). By Proposition \ref{info}(1), $(N')_\fp$  is an  $\R^{(1)}_\fp$-module with the same rank as  $N$; let $e$ be this  rank. Then, \begin{equation}\label{unxeq}
\dim Q(\R^{(1)}/\fp) \otimes_{\R^{(1)}} \R   = \ell((N')_\fp)\gs e  \hspace{0.5cm}\text{ (Proposition \ref{posAS})}.
\end{equation}
Therefore, by Lemma \ref{posFiber}  there exist  $P_0\subsetneq \cdots \subsetneq P_{e-1}$ in $\Proj^2 \R$ with
$P_i\cap \R^{(1)}  = \fp$ for every $0\ls i\ls e-1$. We have an inclusion of domains $A=\R^{(1)}/\fp \hookrightarrow B:=\R/P_0$ and  Lemma \ref{posFiber}  implies $\dim Q(A )\otimes_A B\gs e$. Hence,  $\dim Q(A /\fp')\otimes_A B\gs e$ for every $\fp'\in \Proj^1 A$ (\cite[14.8(b)]{E}).  Choose a $\fp'$ that avoids a general element of $A$ (cf. \cite[14.5]{E}) and that $\dim A/\fp' = 1$;  such $\fp'$ exists by Hilbert's Nullstellensatz. Additionally, choose $P_{e-1}'\in \Proj^2 B$ such that $P_{e-1}'\cap A=\fp'$ and its  image in $ Q(A /\fp')\otimes_A B$ has height $\gs e-1$ (Lemma \ref{posFiber}). Then, from \cite[14.5]{E} we obtain
 \begin{equation}\label{dimes}
 \dim \Proj^1 A =\dim A_{\fp'} =\dim B_{P_{e-1}'} - \dim  B_{P_{e-1}'}/\fp' B_{P_{e-1}'} \ls \dim \Proj^2 B -e+1.
 \end{equation}
Thus,
$$\ell(M\oplus N)=\dim \R\otimes_R k \gs \dim B\gs  \dim \Proj^2 B +2\gs  \dim \Proj^1 A+e+1 = \ell(M)+e.$$
Where the second inequality follows from Lemma \ref{dims}(1). Likewise, $\ell(M\oplus N)\gs \ell(N)+\rank (M)$,  finishing the proof.
 %
\end{proof}

\begin{rem}\label{decomRemark}
 Under the conditions of Theorem \ref{mainAS}, let us assume $R$ has infinite residue field. Then $\ell(M)=\ell(\overline M)$. Hence Theorem \ref{mainAS} applies when $\overline{M}$ is decomposable. 
\end{rem}

In the following corollary we observe that if a module satisfies equality in one of the inequalities in Proposition \ref{posAS}, then we obtain a closed formula for the analytic spread of its direct sum with any other module.

\begin{cor}\label{cormainAS}
Let $M$ and $N$ be $R$-modules having a rank.
\begin{enumerate}
\item[$(1)$] If $\ell(N)=\rank(N)$, then $\ell(M\oplus N)=\ell(M)+\rank(N)$.
\item[$(2)$] If $\ell(N)=\dim R+\rank(N)-1$, then
$\ell(M\oplus N)=\dim R+\rank (M)+\rank (N)-1$.\end{enumerate}
\end{cor}
\begin{proof}
The conclusion follows from Theorem \ref{mainAS}, Proposition \ref{posAS}, and the fact that $\rank(M\oplus N)=\rank(M)+\rank (N)$.
\end{proof}

\begin{rem}
We note that the equality $\ell(N)=\dim R+\rank(N)-1$ is satisfied in a variety of situations. For example, if $N$ is torsion-free and $F/N$ has finite length for some free $R$-module $F$ (\cite[8.4]{V}); if $N$ is an {\it ideal module} (i.e., $N$ is torsion-free and $\Hom_R(N,R)$ is free), and such that $N_\fp$ is free for any $\fp\in \Spec(R)\setminus \{\fm\}$ (\cite[5.2]{SUV}); and if $R$ is a two-dimensional local normal domain with infinite residue field and $N$ is not free (\cite[page 418]{V}).

 The equality $\ell(N)=\rank(N)$ trivially holds for any free $R$-module.
\end{rem}

In the following corollary we relate the analytic spread of direct sums and products of ideals and modules. We remark that the estimates for the analytic spread  in \cite[6.5 6.8]{BS1} and \cite[5.9]{BS2} follow from our next result.

\begin{cor}\label{ideals}
Let $I_1,\ldots, I_{p-1}$ be $R$-ideals  for some $p\gs 1$ and let $M$ be an $R$-module, all of positive rank. Then
\begin{align*}
\ell(I_1\cdots I_{p-1} M)+p-1&=\ell(I_1\oplus \cdots \oplus I_{p-1}\oplus M)\\
&\gs \max_{1\ls i\ls p-1}\{\ell(I_i)+\rank (M)-1,\ell(M)\}+p-1.
\end{align*}
\end{cor}
\begin{proof}
As the inequality follows directly from Theorem \ref{mainAS}, it suffices to show the equality.

We may assume $M$ is torsion-free  (Remark \ref{torsionless}).
We proceed by induction on $p\gs 1$, the case $p=1$ being  clear. Now, assume $p\gs 2$ 
and set $A=\FF(I_1\oplus \cdots \oplus I_{p-1}\oplus M)$. Notice that   $A$  has a natural $\NN^{p}$-graded structure (Definition \ref{grading}). Moreover,  $A^{\Delta}=\FF(I_1\cdots I_{p-1}M)$,  $A^{(i)}=\FF(I_1\oplus\cdots\oplus I_{i-1}\oplus I_{i+1} \oplus\cdots \oplus I_{p-1}\oplus M)$ for every $1\ls i\ls p-1$, and $A^{(p)}=\FF(I_1\oplus \cdots \oplus I_{p-1})$. Hence,  $\dim A^{(i)} <\dim A$ for every $1\ls i\ls p$ (Theorem \ref{mainAS}). Therefore, by Lemmas \ref{dims}(2) and \ref{bijection} we have
$$\dim A = \dim \Proj^{p} A + p  =\dim \Proj^{1} A^\Delta + p  = \dim A^\Delta +p-1,$$ and the result follows.
\end{proof}

The following example extends \cite[8.6]{V}. Here we are able to provide a formula for the analytic spread of a certain class of modules.

\begin{ex}
Let $A_1,\ldots, A_p$ be standard graded $k$-algebras and  for each $i=1,\ldots,p $ let $I_i$ be an $R_i$-ideal of positive rank and generated by elements of degree $\delta_i$. Consider $A=A_1\otimes_k \cdots \otimes_k A_p$ and identify each $I_i$ with its image in $A$. Then   $I_1\cdots I_p$ is generated in degree $\delta_1+\cdots +\delta_p$ and its  minimal number of generators is the dimension of the $k$-vector space $[I_1]_{\delta_1}\otimes_k\cdots \otimes_k
 [I_p]_{\delta_p}$, i.e., $\prod_{i=1}^p \dim_k [I_i]_{\delta_i}=\prod_{i=1}^p\mu(I_i)$, where $\mu(-)$ denotes minimal number of generators. Likewise, for every $n\in \NN$, we have $\mu((I_1\cdots I_p)^n)=\prod_{i=1}^p\mu(I_i^n)$. Hence, $\ell(I_1\cdots I_p)-1=(\ell(I_1)-1)+\cdots+(\ell(I_p)-1)$ (\cite[4.1.3]{BH}). From Corollary \ref{ideals} we conclude that
$$\ell(I_1\oplus\cdots \oplus I_p)=\ell(I_1)+\cdots+\ell(I_p).$$
\end{ex}

In the following corollary we recover, and slightly extend, 
the results in \cite[5.5]{Lipman} (see also \cite[8.4.4]{HS}) and \cite[Lemma 4.7]{Hyr02}). We  recall that the analytic spread of an ideal is defined as $\ell(I)=\dim \R(I)\otimes_R k$ regardless of any rank assumption. 

\begin{cor}
Let $I$ and $J$ be $R$-ideals (not necessarily with a rank).  Then
\begin{enumerate}
\item[$(1)$]  If $I$ or $J$ is not nilpotent, then
$\ell(I)+\ell(J)>\ell(IJ)$.
\item[$(2)$]  If $IJ$ has positive height, or $\sqrt{I}=\sqrt{J}$, then $\ell(IJ)\gs \max\{\ell(I),\ell(J)\}$.
\end{enumerate}
\end{cor}
\begin{proof}
For (1), assume $I$ is not nilpotent. If $J$ is nilpotent, i.e., $\ell(J)=0$, the inequality clearly holds. Otherwise, for any $\fp$ minimal prime of $R$ that does not contain $IJ$ we have $$\ell(I)+\ell(J)\gs \ell(I(R/\fp))+\ell(J(R/\fp))>\ell(IJ(R/\fp))$$
where the first inequality follows from \cite[5.1.7]{HS} and the second one from Theorem \ref{mainAS} and Corollary \ref{ideals}. The result then follows from  \cite[5.1.7]{HS}. Similarly, for (2), let $\fp$ be a minimal prime of $R$ such that $\ell(I)=\ell(I(R/\fp))$ (\cite[5.1.7]{HS} ), then $$\ell(IJ)\gs \ell(IJ(R/\fp))\gs \ell(I(R/\fp))=\ell(I),$$
where the second inequality follows from Corollary \ref{ideals}. Likewise, $\ell(IJ)\gs \ell(J)$, and the result follows.
\end{proof}

Our results allow us to build a minimal reduction of a direct sum of multiple copies of an ideal $I$ as we show in the next corollary. This result extends \cite[8.67]{V} to arbitrary ideals. Moreover, the computation of integral closure in \cite[3.5]{Kod} follows from this result.

Given elements $a_1,\dots, a_s\in R$ and an integer $p\geq 1$, we define the matrix
$$
A^p(a_1,\ldots, a_s):=\begin{pmatrix}
a_1&a_2&a_3&\cdots &0&0&0\\
0&a_1&a_2&\cdots &0&0&0\\
\vdots&\vdots&\vdots&\ddots&\vdots&\vdots&\vdots\\
0&0&0&\cdots&a_{s-1}&a_s&0\\
0&0&0&\cdots&a_{s-2}&a_{s-1}&a_s
\end{pmatrix}.
$$

\begin{cor}\label{reducDSum}
Let $I$ be an $R$-ideal of positive rank  and let $s$ be its  analytic spread. Fix $p\in \ZZ_{>0}$ and consider the $R$-module $$M=\underbrace{I\oplus\cdots\oplus I}_{p   \text{ times}}.$$ Then, $\ell(M)=s+p-1$ and given any (minimal) reduction  $(a_1,\ldots, a_s)\subseteq I$, the $R$-submodule of $R^p$ generated by the columns of the matrix $A^p(a_1,\ldots, a_s)$ is a minimal reduction of $M$.
\end{cor}
\begin{proof}
Let $U$ be  the module generated by the columns of this matrix and notice that $U\subseteq M$. We first show that $U$ is a reduction of $M$. For this, note that by \cite[page 15]{BV}, $\I_p(U)=I^p$, and the latter is clearly also equal to $\I_p(M)$. By Theorem \ref{icm}, 
 it follows that $U$ is a reduction of $M$.

It remains to show $\ell(M)=s+p-1$, but this follows from  Corollary \ref{ideals} since $$\ell(M)=\ell(I^p)+p-1=s+p-1,$$
finishing the proof.
\end{proof}

\begin{ex}
Let $I$ be a monomial ideal of $\O_2$.
Let $\Gamma_+(I)$ denote the Newton polyhedron of $I$ (see the definition of this notion before Example \ref{exidcd})
and let $\{(a_1,b_1),\, (a_2,b_2),\,\ldots,\, (a_n,b_n)\}\subset \NN^2$
be the set of vertices of $\Gamma_+(I)$, with $n\geq 2$ and $a_1<a_2<\cdots<a_n$ and $b_1>b_2>\cdots>b_n$. Consider the polynomials
of $\C[x,y]$ given by
$$
g_1=\displaystyle\sum_{i \text{ is  odd}}x^{a_i}y^{b_i} \hspace{0.7cm} \textnormal{and}\hspace{0.7cm} g_2=\displaystyle\sum_{i \text{ is  even}}x^{a_i}y^{b_i}.$$
By \cite{BFS} (see also \cite[3.6]{CQ} or \cite[3.7]{CL}), the ideal $(g_1,g_2)$ is a reduction of $I$. Thus, by Corollary \ref{reducDSum}, the module generated by the columns of $A^p(g_1,g_2)$ is a minimal reduction of the module $M=I\oplus\cdots \oplus I\subset \O_2^p$.
\end{ex}


\section{Integrally decomposable modules, Newton non-degeneracy, and the computation of the integral closure}\label{WD}

In this section we address the task of computing the integral closure of modules.  In general, this is  a difficult and involved process as it requires the computation of the normalization of Rees algebras. In our main results we focus on a  wide family of modules, that we call {\it integrally decomposable}, for which an important example are the {\it Newton non-degenerate} modules (see Definitions \ref{WDdef} and \ref{NNDm}). In our main results, we express the integral closure of these modules in terms of the integral closure of  its component ideals (see  Theorem \ref{WDcentral} and Corollary \ref{icmcentral}). Therefore, we translate the problem of computing integral closures of modules to integral closures of ideals, for which several algorithms are available in the literature (see for instance  \cite[Chapter 6]{V2}). 

Throughout this section $R$ is a Noetherian ring.

\subsection{Integrally decomposable modules}

Let $M$ be a submodule of $R^p$ and let $r=\rank(M)$. We identify $M$ with any matrix of generators and denote by $\Lambda_M$ the set of vectors $(i_1,\dots, i_r)\in \Z^r_{>0}$ such that
$1\leq i_1<\cdots <i_r\leq p$ and there exists some non-zero minor of $M$ formed from rows $i_1,\dots, i_r$.

\begin{defn}\label{WDdef}
Let $M$ be submodule of $R^p$ and let $r=\rank(M)$. We say that $M$ is {\it integrally decomposable} when
$\overline{M_\tL}$ is decomposable, for all $\tL\in \Lambda_M$.

\end{defn}

We remark that, under the conditions of the above definition, if $\tL\in\Lambda_M$ and
we write $\tL=(i_1,\dots, i_r)$, where $1\leq i_1<\cdots <i_r\leq p$, then
$\overline{M_\tL}$ is decomposable if and only if $\overline{M_\tL}=(\overline{M_\tL})_{i_1}
\oplus \cdots\oplus (\overline{M_\tL})_{i_r}$. In particular, we observe that Definition \ref{WDdef} constitutes a void condition
when $\rank(M)=1$.

\begin{lem}\label{Mibarra}
Let $M$ be submodule of $R^p$. Then $\overline{(\overline M)_i}=\overline{M_i}$, for all $i=1,\dots, p$.
\end{lem}

\begin{proof}
Fix an index $i\in\{1,\dots, p\}$.
The inclusion $M\subseteq \overline M$ implies that $M_i\subseteq (\overline M)_i$. Thus $\overline{M_i}\subseteq \overline{(\overline M)_i}$.
From Proposition \ref{projections} we deduce that $(\overline M)_i\subseteq \overline{M_i}$. Therefore,
$\overline{(\overline M)_i}\subseteq \overline{M_i}$, and hence the result follows.
\end{proof}

\begin{prop}\label{lemaid}
Let $M$ be submodule of $R^p$ and let $r=\rank(M)$. Then
$M$ is integrally decomposable if and only if
\begin{equation}\label{mbarradec}
\overline {M_\tL}=\overline{M_{i_1}}\oplus \cdots \oplus \overline{M_{i_r}}.
\end{equation}
for all $\tL=(i_1,\dots,i_r)\in\Lambda_M$, where $1\leq i_1<\cdots<i_r\leq p$.
\end{prop}

\begin{proof}
Since $M_\tL$ is a submodule of $R^r$ of rank $r$, for all $\tL\in\Lambda_M$, it suffices to show the result in the case $r=p$.
So let us assume that $\rank(M)=p$.
In general we have the following inclusions:
$$
 \overline M\subseteq (\overline M)_1\oplus \cdots \oplus (\overline M)_p\subseteq
\overline{(\overline M)_1}\oplus \cdots \oplus \overline{(\overline M)_p}=\overline{M_1}\oplus \cdots \oplus \overline{M_p}
$$
where the last equality is an application of Lemma \ref{Mibarra}.
This shows that if relation (\ref{mbarradec}) holds, then $\overline M$ is decomposable.

Conversely, if $\overline M$ is decomposable, then $\overline M=(\overline M)_1\oplus \cdots \oplus (\overline M)_p$. Taking integral closures in this equality
it follows that
$$
\overline M=\overline{\overline M}=\overline{(\overline M)_1\oplus \cdots \oplus (\overline M)_p}=
\overline{(\overline M)_1}\oplus \cdots \oplus \overline{(\overline M)_p}=\overline{M_1}\oplus \cdots \oplus \overline{M_p}
$$
again by Lemma \ref{Mibarra}, and thus equality (\ref{mbarradec}) follows.
\end{proof}

In the following proposition we characterize integrally decomposable modules in terms of their ideals of minors.





\begin{prop}\label{WDprop}
Let  $M$ be  a submodule of $R^p$ and let $r=\rank(M)$.  Then the following conditions are equivalent.
\begin{enumerate}
\item[$(1)$] $M$ is integrally decomposable.
\item[$(2)$] $\overline{\I_r(M_\tL)}=\overline{\prod_{i\in \tL}M_i}$, for all $\tL\in \Lambda_M$.
\end{enumerate}
\end{prop}

\begin{proof}
Fix $\tL=(i_1,\ldots, i_r)\in \Lambda_M$ and let $N=M_{i_1}\oplus \cdots \oplus M_{i_r}$. Then
$\overline{M_\tL}\subseteq \overline{N} = \overline{M_{i_1}}\oplus \cdots \oplus \overline{M_{i_r}},$
where the last equality holds by  Remark \ref{basics}.
Therefore, by Theorem \ref{icm}, $\overline{M_\tL}=\overline{N}$ if and only if
$\overline{\I_r(M_\tL)}=\overline{\I_r(N)}=\overline{M_{i_1}\cdots M_{i_r}}$. Then the result follows as a direct application of Proposition \ref{lemaid}.
\end{proof}

Let $R$ be a Noetherian local ring of dimension $d$ and let
$I_1,\dots, I_d$ be a family of ideals of $R$ of finite colength.
We denote by $e(I_1,\dots, I_d)$ the mixed multiplicity of the family of ideals $I_1,\dots, I_d$ (see \cite[p.\,339]{HS}).
We recall that when the ideals $I_1,\dots, I_d$ coincide with a given ideal $I$ of finite colength, then
$e(I_1,\dots, I_d)=e(I)$, where $e(I)$ is the multiplicity of $I$, in the usual sense.

Let $(i_1,\dots, i_p)\in\Z_{\geq 0}^p$, for some $p\leq d$, such that $i_1+\cdots+i_p=d$.
We denote by
$e_{i_1,\dots, i_p}(I_1,\dots, I_p)$ the mixed multiplicity $e(I_1,\dots, I_1,\dots, I_p,\dots, I_p)$ where
$I_j$ is repeated $i_j$ times, for all $j=1,\dots, p$.

Let $M$ be a submodule of $R^p$ of finite colength. Following \cite[p.\,418]{BiviaJLMS}, we define
$$
\delta(M)=\sum_{\substack{i_1+\cdots +i_p=d\\ i_1,\dots,
i_p\geq 0 }}e_{i_1,\dots, i_p}(M_1,\dots, M_p).
$$
We remark that the condition that $M$ has finite colength in $R^p$ implies that $M_i$ has finite colength in $R$, for all $i=1,\dots, p$.

By a result of Kirby and Rees in \cite[p.\,444]{Kirby-Rees} (see also \cite[p.\,417]{BiviaJLMS}),
we have that $e(I_1\oplus \cdots \oplus I_p)=\delta(I_1\oplus \cdots \oplus I_p)$, for any family of ideals
$I_1,\dots,I_p$ of $R$ of finite colength. Therefore $\delta(M)=e(M_1\oplus \cdots \oplus M_p)$.

\begin{prop}\label{idcodfinita}
Let $R$ be a formally equidimensional Noetherian local ring of dimension $d>0$.
Let $M$ be a submodule of $R^p$. Let $r=\rank(M)$. Assume $M_\tL$ has finite colength, as a submodule of $R^{r}$,
for all $\tL\in \Lambda_M$. Then $M$ is integrally decomposable if and only if $e(M_\tL)=\delta(M_\tL)$, for all $\tL\in \Lambda_M$.
\end{prop}

\begin{proof}
Let us fix any $\tL=(i_1,\dots, i_r)\in\Lambda_M$.
By Proposition \ref{lemaid}, the submodule $\overline{M_\tL}\subseteq R^r$ is decomposable if and only if
$\overline{M_\tL}=\overline{M_{i_1}}\oplus \cdots \oplus \overline{M_{i_r}}$.
Let us recall that
$$
\overline{M_{i_1}}\oplus \cdots \oplus \overline{M_{i_r}}= \overline{\overline{M_{i_1}}\oplus \cdots \oplus \overline{M_{i_r}}}.
$$
Thus $M_\tL$ is integrally decomposable if and only if $M_\tL$ is a reduction of $\overline{M_{i_1}}\oplus \cdots \oplus \overline{M_{i_r}}$, which is to say
that $e(M_\tL)=e(M_{i_1}\oplus \cdots \oplus M_{i_r})$, by Theorem \ref{numCr}. But $e(M_{i_1}\oplus \cdots \oplus M_{i_r})=\delta(M_\tL)$, thus the result follows.
\end{proof}

For a submodule of $R^p$, we introduce the following objects.

\begin{defn}\label{JMCM}
Let $M\subseteq R^p$ be a submodule of rank $r$. We define the ideal
$$
J_M=\sum_{(i_1,\dots, i_r)\in\Lambda_M}M_{i_1}\cdots M_{i_r}
$$
and the following modules
\begin{align*}
Z(M)&=\left\{ h\in R^p: \rank(M)=\rank(M + R h) \right\}\\
C(M)&=Z(M)\cap \left(\overline{M_1}\oplus \cdots\oplus \overline{M_p}\right).
\end{align*}
\end{defn}

\begin{rem}
In the previous definition, if $r=p$ then $Z(M)=R^p$ and thus $C(M)=\overline{M_1}\oplus \cdots\oplus \overline{M_p}$.
\end{rem}

From Remarks \ref{sameRank} and  \ref{basics} it follows that $\overline M$ is always contained in $C(M)$ but this containment can be strict.
We ask the following question.

\begin{question}\label{QuesCofM}
Let $M$ be  a submodule of $R^p$, when do we have $\overline M=C(M)?$
\end{question}

The following is the main theorem of this section, here we provide a  partial answer to Question \ref{QuesCofM} by showing that   integrally decomposable modules satisfy this equality.

\begin{thm}\label{WDcentral} Let $M$ be a submodule of $R^p$ and let $r=\rank(M)$.
Consider the following conditions.
\begin{enumerate}
\item[$(1)$] $M$ is integrally decomposable.
\item[$(2)$] $\overline{\I_r(M)}=\overline{J_M}$.
\item[$(3)$] $\overline M=C(M)$.
\end{enumerate}
Then $\textnormal{(1)}\Rightarrow\textnormal{(2)}\Rightarrow\textnormal{(3)}$. Moreover, if $r=p$, then these implications become equivalences.
\end{thm}

We remark that $\textnormal{(3)}\nRightarrow\textnormal{(2)}$. In particular
$\textnormal{(3)}\nRightarrow\textnormal{(1)}$ in general. This is shown in Example \ref{CMnotID}.
In a wide variety of examples of modules $M\subseteq R^3$ with $\rank(M)=2$ we have verified that $M$ is integrally decomposable when
$\overline{\I_r(M)}=\overline{J_M}$. However we have not yet found a proof or a counterexample of the implication
$\textnormal{(2)}\Rightarrow\textnormal{(1)}$; we conjecture that this implication holds in general.

We present the proof of Theorem \ref{WDcentral} after the following remark and lemma.

\begin{rem}
We observe that $\overline{\I_r(M)}\subseteq \overline{J_M}$. 
In general, this inclusion might be strict. For instance, consider the submodule $M\subseteq \O_2^3$ generated by the columns of the following matrix
$$
\left[\begin{matrix}
x^2 & xy & x^3 \\
y^2 & y^2 & y^2  \\
x+y & 2y & x^2+y
\end{matrix}\right].
$$
Notice that $M_1=( x^2, xy)$, $M_2=( y^2)$
and $M_3=( x,y)$. We see that $\rank(M)=2$ and
$$
\overline{J_M}=\overline{M_1M_2+M_1M_3+M_2M_3}=\overline{( x^3, y^3)}.
$$
However, $\I_2(M)=(x^2y, xy^2, y^3)$. Therefore $\overline{\I_r(M)}$ is strictly contained in $\overline{J_M}$.
\end{rem}

We need one more lemma prior presenting the proof of the theorem.

\begin{lem}\label{ranks}
Let $M\subseteq R^p$ be a submodule and let $h\in R^p$. If $\rank(M)=\rank(M + R h)$, then
$\rank(M_\tL)=\rank(M_\tL+ R h_\tL)$, for any $\tL\subseteq\{1,\dots, p\}$, $\tL\neq \emptyset$.
\end{lem}

\begin{proof}
Let us identify $M$ with a given matrix of generators.
Let $Q(R)$ denote the total ring of fractions of $R$. We note that  $\rank(M)=\rank(M + R h)$  if and only if $h$ is equal to a linear combination of the columns of $M$ with coefficients  in $Q(R)$. By projecting this linear combination onto the rows corresponding to $\tL$ we obtain that $h_\tL$ is equal to a linear combination of the columns of $M_\tL$, which means $\rank(M_\tL)=\rank(M_\tL + R h_\tL)$, as desired.
\end{proof}

We are now ready to present the proof of Theorem \ref{WDcentral}.

\begin{proof}[Proof of Theorem \ref{WDcentral}]
 We begin with   $\textnormal{(1)}\Rightarrow\textnormal{(2)}$.
From
$
\I_r(M)=\sum_{\tL\in\Lambda_M}\I_r(M_\tL)
$
and Proposition \ref{WDprop} we obtain
\begin{equation*}\label{Irsum}
\overline{\I_r(M)}=\overline{\sum_{\tL\in\Lambda_M} \I_r(M_\tL)}=\overline{\sum_{\tL\in\Lambda_M}\overline{\I_r(M_\tL)}}=\overline{\sum_{\tL\in\Lambda_M}\overline{\prod_{i\in \tL}M_i}}=
\overline{J_M}.
\end{equation*}

We continue with $\textnormal{(2)}\Rightarrow\textnormal{(3)}$.
The inclusion $\overline M\subseteq C(M)$ follows immediately from Remarks \ref{sameRank} and \ref{basics}, then  we need to show  the reverse inclusion.
Let $h\in C(M)$, we claim that  $\I_r(M+R h)\subseteq \overline{\I_r(M)}$. We note that if the claim holds then  $h$ is integral over $M$, by Theorem \ref{icm}, finishing the proof.

Now we prove the claim. 
 Identify $M$ with a matrix of generators and let $g$ be a non-zero minor of size $r$ of the matrix $[M|h]$ with row set $\tL=\{i_1,\dots, i_r\}$.
By Lemma \ref{ranks}, we have $\rank(M_\tL)=\rank(M_\tL|h_\tL)$. In particular, the matrix
$M_\tL$ has some non-zero minor of order $r$. This implies that $\tL\in\Lambda_M$. Since $h\in \overline{M_1}\oplus\cdots\oplus \overline{M_p}$, we have
$g\in \prod_{i\in\tL}\overline{M_i}\subseteq \overline{\prod_{i\in\tL}M_i} \subseteq \overline{J_M}$ (\cite[1.3.2]{HS}). Therefore,
$\I_r(M+R h)\subseteq \overline{J_M}=\overline{\I_r(M)}$, and  the claim follows.

Let us suppose that $r=p$. In this case $C(M)=\overline{M_1}\oplus \cdots \overline{M_p}$ and therefore the equivalence of the conditions
follows as a direct consequence of Propositions \ref{lemaid}. 
\end{proof}


The following result shows a procedure to compute the module $Z(M)$ with the aid of Singular \cite{Singular} or other computational algebra programs.
If $N$ is a submodule of $R^p$, then we denote by $N^T$ the transpose of any matrix whose columns generate $N$.

\begin{lem}\label{syzygies}
Let $R$ be an integral domain and let $M$ be an $p\times m$ matrix with entries in $R$. Then
$$\{h\in R^p: \rank(M)=\rank([M\mid h])\}=\ker\big((\ker(M^T))^T\big).$$
\end{lem}
\begin{proof}
Let $Q(R)$ be the field of fractions of $R$ and let $\ker_{Q(R)}(-)$ the kernel of matrices computed over $Q(R)$.

Clearly the rank of a matrix over $R$ is equal to the rank as a matrix over $Q(R)$. Let $h\in R^p$, then by the dimension theorem for matrices we have
$$\rank(M)=\rank(M^T)=p-\dim \ker_{Q(R)}(M^T),\quad \text{and}\quad \rank([M\mid h])=p-\dim \ker_{Q(R)}([M\mid h]^T).$$
Since we always have $\ker_{Q(R)}([M\mid h]^T) \subseteq \ker_{Q(R)}(M^T)$, it follows that
\begin{align*}
\rank(M)=\rank([M\mid h]) & \Longleftrightarrow \ker_{Q(R)}([M\mid h]^T) = \ker_{Q(R)}(M^T)\\
 &\Longleftrightarrow h^Tv=0  \text { for every } v\in  \ker_{Q(R)}(M^T)\\
 & \Longleftrightarrow h^Tv=0 \text{ for every } v\in  \ker(M^T)\\
 & \Longleftrightarrow h\in \ker\big((\ker(M^T))^T\big).
\end{align*}
This finishes the proof.
\end{proof}

\begin{rem}\label{sobreSing1}
Given a submodule $M$ of $R^p$, the computation of $Z(M)$ can be done with Singular \cite{Singular} as follows.
Denoting also by $M$ a matrix whose columns generate this module, then $Z(M)$ is generated by the columns of the matrix obtained as
\texttt{syz(transpose(syz(transpose(M))))}.
\end{rem}

In the next example we show an application of Theorem \ref{WDcentral} in order to compute the integral closure of a module.
First, we introduce some concepts.

Let us fix coordinates $x_1,\dots, x_n$ for $\C^n$. If $n=2$, we simply write $x,y$ instead of $x_1, x_2$. If $\mathbf{k}=(k_1,\dots, k_n)\in\N^n$, then we denote the monomial $x_1^{k_1}\cdots x_n^{k_n}$ by $x^\mathbf{k}$.  If $f\in\O_n$ and $f=\sum_{\mathbf k} a_\mathbf k x^\mathbf k$ is the Taylor expansion of $f$ around the origin, then the {\it support of $f$}, denoted by $\supp(f)$, is the set
$\{\mathbf k\in\N^n: a_\mathbf k\neq 0\}$.
The {\it support} of  a non-zero ideal $I$ of $\O_n$ is the union of the supports of the elements of $I$.
We denote this set by $\supp(I)$.

Given a subset $A\subseteq \mathbb R^n_{\geq 0}$, the {\it Newton polyhedron} determined by $A$, denoted by $\Gamma_+(A)$,
is the convex hull of the set $\{\mathbf k+\mathbf v: \mathbf k\in A, \mathbf v\in\mathbb R^n_{\geq 0}\}$.
The {\it Newton polyhedron of $f$} is defined as $\Gamma_+(f)=\Gamma_+(\supp(f))$.
 For an ideal $I$ of $\O_n$, the {\it Newton polyhedron of $I$} is defined as $\Gamma_+(I)=\Gamma_+(\supp(I))$.
 It is well-known that $\Gamma_+(I)=\Gamma_+(\overline I)$ (see for instance \cite[p.\,58]{BFS}).

Let $\w\in\Z^n_{\geq 0}$ and let $f\in\O_n$, $f\neq 0$. We define
$d_\w(f)=\min\{\langle \w,\mathbf k\rangle: \mathbf k\in\supp(f)\}$, where
$\langle \w,\mathbf k\rangle$ denotes the usual scalar product. If $f=0$ then we set $d_\w(f)=+\infty$.
We say that a non-zero $f\in\O_n$ is {\it weighted homogeneous with respect to $\w$} when
$\langle \w,\mathbf k\rangle=d_\w(f)$, for all $\k\in\supp(f)$.

\begin{ex}\label{exidcd}
Let us consider the submodule $M$ of $\O_2^3$ generated by the columns of the following matrix:
$$
\left[\begin{matrix}
x^2y & xy^3 & x^2+y^5 \\
xy^3 & x^2+y^5 & x^2y \\
x^2y-xy^3 & xy^3-x^2-y^5 & x^2+y^5-x^2y
\end{matrix}\right].
$$
We observe that $\rank(M)=2$ and $\Lambda_M=\{(1,2),(1,3),(2,3)\}$.
Using Singular \cite{Singular}
we verified  that $M_{12}$, $M_{13}$ and $M_{23}$ have finite colength and $e(M_{12})=e(M_{13})=e(M_{23})=33$.

Let $I=M_1=M_2=(x^2y, xy^3, x^2+y^5)$. We have $e(I)=11=e(M_3)$. Since $M_3\subseteq I$, it follows that
$\overline I=\overline{M_3}$. Hence $e(M_1, M_2)=e(M_1, M_3)=e(M_2, M_3)=e(I)=11$. This fact shows that
$\delta(M_{12})=e(M_1)+e(M_1,M_2)+e(M_2)=3e(I)=33=\delta(M_{13})=\delta(M_{23})$.
Therefore $M$ is integrally decomposable, by Proposition \ref{idcodfinita}.

By Theorem \ref{WDcentral}, the integral closure of $M$ is expressed as
$$
\overline M=\left\{h
\in \overline I\oplus \overline I \oplus\overline I: \rank(M+\O_2h)=2\right\}.
$$
Let $L=(x^2+y^5, xy^3, x^2y, x^3, y^6)$. We observe that $IL=L^2$, therefore $I$ is a reduction of $L$. Hence
$L\subseteq \overline I$. Let us see that equality holds.

Let $f=x^2+y^5$. We observe that $f$ is weighted homogeneous with respect to $\w=(5,2)$.
Let $N$ denote the ideal of $\O_2$ generated by all monomials $x^{k_1}y^{k_2}$, where $k_1, k_2\in\Z_{\geq 0}$, such that
$d_\w(x^{k_1}y^{k_2})=5k_1+2k_2\geq 11$. Then $L=(f)+N$.

Let $g\in \overline I$. In particular $\Gamma_+(g)\subseteq \Gamma_+(\overline I)=\Gamma_+(I)=\Gamma_+(x^2, y^5)$.
Let $g_1$ denote the part of lowest degree with respect to $w$ in the Taylor expansion of $g$, and let $g_2=g-g_1$.
Then $d_\w(g_1)\geq 10$ and $d_\w(g_2)\geq 11$. In particular $g_2\in N\subseteq L$. Then $g\in L$ if and only if $g_1\in L$.

We may assume that $\supp(g_1)\subseteq \{(2,0), (0,5)\}$, as otherwise $g\in L$. If $\supp(g_1)$ is equal to $\{(2,0)\}$ or to $\{(0,5)\}$, then the ideal
$(f, g_1)$ has finite colength and $e(f, g_1)=10$, which is a contradiction, since $(f,g_1)\subseteq \overline I$ and $e(I)=11$.
Therefore $g_1=\alpha x^2+\beta y^5$, for some $\alpha, \beta \in \C\setminus\{0\}$.
If $\alpha\neq \beta$, we would have that $(f,g_1)$ is an ideal of finite colength and $e(f, g_1)=10$. Therefore $\alpha=\beta$, which means
that $g_1\in (f)\subseteq L$. Therefore $\overline I\subseteq L$.

By Theorem \ref{WDcentral}, we have that $\overline M=Z(M)\cap (\overline I\oplus \overline I\oplus \overline I)$.
The module $Z(M)$ can be computed by means of Lemma \ref{ranks}. Thus we obtain that $Z(M)$ is generated by the columns of the matrix
$$
\left[\begin{array}{rr}
1 & 0 \\
0 & -1 \\
1 & 1
\end{array}\right].
$$
We have seen before that $\overline I=L$. Let us remark that $\{x^2+y^5, xy^3, y^6\}$ is a minimal system of generators of $L$.
Then, by intersecting the modules $Z(M)$ and $L\oplus L\oplus L$, we finally obtain that $\overline M$ is generated by the columns of the following matrix:
$$
\left[\begin{array}{cccccc}
x^2+y^5 & xy^3 & y^6 & x^2+y^5 & xy^3 & y^6 \\
x^2+y^5 & xy^3 & y^6 & 0       & 0    & 0  \\
0       & 0    & 0   & x^2+y^5 & xy^3 & y^6
\end{array}\right].
$$

\end{ex}

In the next subsection we will introduce an important class of modules that are integrally decomposable.

\subsection{Newton non-degenerate modules}
Let us fix coordinates $x_1,\dots, x_n$ for $\C^n$.
Let $M$ be a submodule of $\O_n^p$ and let us identify $M$ with any  matrix of generators of $M$.
We recall that $M_i$ is the ideal of $\O_n$ generated by the elements of $i$-th row of $M$. We define the {\it Newton polyhedron of $M$} as 
$$
\Gamma_+(M)=\Gamma_+\big(\prod_{i=1}^pM_i\big)=\Gamma_+(M_1)+\cdots+\Gamma_+(M_p)=\{\mathbf k_1+\cdots+\mathbf k_p: \mathbf k_i\in \Gamma_+(M_i),\,\,\textnormal{for all }i\}.
$$
We denote by $\mathscr F_c(\Gamma_+(M))$ the set of compact faces of $\Gamma_+(M)$ (see \cite[p.\,408]{BiviaJLMS} or \cite[p.\,397]{BiviaMRL} for details).

Let $I$ be  an ideal of $\O_n$. We denote by $I^0$ the ideal by all monomials $x^\mathbf k$ such that $\mathbf k\in\Gamma_+(I)$. We refer to this ideal as the {\it term ideal} of $I$. If $I$ is the zero ideal, then we set $\Gamma_+(I)=\emptyset$ and $I^0=0$. Recall that an ideal is said to be {\it monomial} if it admits a generating system formed by monomials. It is known that if $I$ is a monomial ideal, then $\overline I=I^0$ (see \cite[p.\,141]{E}, \cite[p.\,11]{HS}, or \cite[p.\,219]{T2}). 
 The ideals $I$ for which $\overline I$ is generated by monomials are characterized in \cite{S} and are called {\it Newton
non-degenerate ideals} (see also \cite{BiviaMRL}, \cite{BFS}, or \cite[p.\,242]{T2}). 

In \cite{BiviaJLMS}, the first author introduced and  studied the notion of Newton non-degenerate modules of maximal rank. Here we extend this concept to modules of submaximal rank. 


Let $f\in\O_n$ and let $f=\sum_{\mathbf k} a_\mathbf k x^\mathbf k$ be the Taylor expansion of $f$ around the origin.
If $\Delta$ is any compact subset of $\mathbb R^n_{\geq 0}$, then we denote by $f_\Delta$ the
polynomial resulting as the sum of all terms $a_\mathbf k x^\mathbf k$ such that $\mathbf k\in \Delta$. If no such $\mathbf k$ exist, then
we set $f_\Delta=0$.

\begin{defn}\label{NNDm}
Let $M$ be a non-zero submodule of $\O_n^p$ and let $r=\rank(M)$. Let $[m_{ij}]$ be a $p\times m$  matrix of generators of $M$, where $p\leq m$.
\begin{enumerate}
\item (\cite[3.6]{BiviaJLMS}) First assume $r=p$. We say that $M$ is {\it Newton non-degenerate} when
$$
\big\{x\in \C^n: \rank[(m_{ij})_{\Delta_i}(x)]<p\big\}\subseteq\big\{x\in\C^n: x_1\cdots x_n=0\big\},
$$
for any $\Delta\in\mathscr F_c(\Gamma_+(M))$, where we write $\Delta$ as
$\Delta=\Delta_1+\cdots+\Delta_p$ with $\Delta_i$ being a compact face of $\Gamma_+(M_i)$, for all $i=1,\dots, p$.

\item Now assume $r<p$. We say that $M$ is {\it Newton non-degenerate} when $M_\tL$
 is Newton non-degenerate, as a submodule (of rank $r$) of $\O_n^r$, for any $\tL\in \Lambda_M$.
\end{enumerate}
\end{defn}

In particular, if $I$ is an ideal of $\O_n$ and $g_1,\dots, g_s$ denotes a generating system of $I$, then $I$ is Newton non-degenerate if and only if
$\{x\in \C^n: (g_1)_\Delta(x)=\cdots=(g_s)_\Delta(x)=0\}\subseteq \{x\in\C^n: x_1\cdots x_n=0\}$, for any $\Delta\in\mathscr F_c(\Gamma_+(I))$.

The following result follows from \cite[3.7, 3.8]{{BiviaJLMS}} and it characterizes the Newton non-degeneracy
of submodules of $\O_n^p$ of maximal rank.

\begin{thm}\label{resultatJLMS}\cite{BiviaJLMS} Let $M\subseteq \O_n^p$ be a submodule of rank $p$. Then the following
conditions are equivalent:
\begin{enumerate}
\item[$(1)$] $M$ is Newton non-degenerate.
\item[$(2)$] $\I_p(M)$ is a Newton non-degenerate ideal and $\Gamma_+(\I_p(M))=\Gamma_+(M)$.
\item[$(3)$] $\overline M=M_1^0\oplus \cdots\oplus M_p^0$.
\end{enumerate}
If furthermore, $\lambda(\O_n^p/M)<\infty$, then the previous conditions are equivalent to the following:
\begin{enumerate}
\item[$(4)$] $e\big(\I_p(M)\big)=n!\mathrm V_n\big(\Gamma_+(M)\big)$.
\item[$(5)$] $M_i$ is Newton non-degenerate, for all $i=1,\dots, p$, and $e(M)=\delta(M)$.
\end{enumerate}

\end{thm}

As an immediate consequence of Theorem \ref{resultatJLMS} the following result follows.

\begin{cor}
\label{carNNDm}
Let $M$ be a submodule of $\O_n^p$ and let $r=\rank(M)$. Then the following conditions are equivalent:
\begin{enumerate}
\item[$(1)$] $M$ is Newton non-degenerate.
\item[$(2)$] $\overline{\I_r(M_\tL)}=\overline{\prod_{i\in \tL}M_i^0}$, for all $\tL\in \Lambda_M$.
\item[$(3)$] $\overline{M_{\{i_1,\dots, i_r\}}}=M^0_{i_1}\oplus \cdots \oplus M^0_{i_r}$, for all $(i_1,\dots, i_r)\in \Lambda_M$.
\item[$(4)$] $M$ is integrally decomposable and $M_i$ is Newton non-degenerate, for all $i=1,\dots, p$.
\end{enumerate}
\end{cor}

Therefore, we see from the previous result that if $M$ is Newton non-degenerate, then it is integrally decomposable. The converse
does not hold in general, as Example \ref{exidcd} shows.



From the results of the previous section we obtain the following combinatorial interpretation for the analytic spread of Newton non-degenerate modules of maximal rank.

\begin{cor}\label{monModAS}
Let $M\subseteq \O_n^p$ be a Newton non-degenerate module of rank $p$, then
$$\ell(M)=\max\big\{\dim(\Delta): \Delta\in \mathscr F_c(\Gamma_+(M))\big\}+p.$$
\end{cor}
\begin{proof}
We may assume $R$ has infinite residue field and then $\ell(M)=\ell(\overline M)$ (see Remark \ref{decomRemark}).
Moreover $\overline M=M_1^0\oplus\cdots\oplus M_p^0$, since $M$ is Newton non-degenerate.
Therefore $\ell(M)=\ell(M_1^0\oplus\cdots\oplus M_p^0)=\ell(M_1^0\cdots M_p^0)+p-1$, where the last equality is an application of
Corollary \ref{ideals}. By \cite[Theorem 2.3]{Bivia03} we have
$$
\ell(M_1^0\cdots M_p^0)=\max\big\{\dim(\Delta): \Delta\in \mathscr F_c\left(\Gamma_+(M_1^0\cdots M_p^0)\right)\big\}+1.
$$
Since $\Gamma_+(M)=\Gamma_+(M_1^0\cdots M_p^0)$ the result follows.
\end{proof}



\begin{ex}
Let $M$ be the submodule of $\O_2^2$ generated by the columns of the following matrix
$$
M=\left[\begin{matrix}
x^3 & xy & y^3 & y^3 \\
x^5 & x^2y & xy^2 & x^5+x^2y
\end{matrix}\right].
$$
We observe that $\rank(M)=2$ and $\I_2(M)$ is a Newton non-degenerate ideal. Moreover $\overline{\I_2(M)}=\overline{(xy^5, x^2y^3, x^3y^2, x^5y, x^8)}=\overline{M_1M_2}.$  
Therefore $\overline M=M_1^0\oplus M_2^0$, by Corollary \ref{carNNDm} and $\ell(M)=\ell(\overline{M})=3$, by Corollary \ref{monModAS}. 
\end{ex}

Analogously to Definition \ref{JMCM}, for a submodule of $\O_n^p$ we introduce the following objects.

\begin{defn}
Let $M\subseteq \O^n_p$ and let $r=\rank(M)$. We define
$$
H_M=\sum_{(i_1,\dots, i_r)\in\Lambda_M}M_{i_1}^0\cdots M_{i_r}^0,
$$
and
$$
C^0(M)=Z(M)\cap \big(M_1^0\oplus \cdots\oplus M_p^0\big),
$$
where we recall that $Z(M)=\left\{ h\in R^p: \rank(M)=\rank(M + R h) \right\}$.
\end{defn}

We remark that $H_M$ is a monomial ideal and $\Gamma_+(H_M)=\Gamma_+(J_M)$. Therefore $\overline{H_M}=J_M^0$,
where $J_M^0$ is the ideal of $\O_n$ generated by the monomials $x^k$ such that $k\in\Gamma_+(J_M)$.
We also remark that $C(M)\subseteq C^0(M)$.
The following result follows from Theorem \ref{WDcentral} and Corollary \ref{carNNDm}.

\begin{cor}\label{icmcentral} Let $M$ be a submodule of $\O_n^p$ and let $r=\rank(M)$. 
Consider the following conditions.
\begin{enumerate}
\item[$(1)$] $M$ is Newton non-degenerate.
\item[$(2)$] $\overline{\I_r(M)}=J_M^0$.
\item[$(3)$] $\overline M=C^0(M)$.
\end{enumerate}
Then $\textnormal{(1)}\Rightarrow\textnormal{(2)}\Rightarrow\textnormal{(3)}$. Moreover, if $r=p$, then these implications become equivalences.
\end{cor}

\begin{rem}\begin{enumerate}
\item The implication $\textnormal{(3)}\Rightarrow\textnormal{(2)}$ in Corollary \ref{icmcentral} does not hold in general, as shown in
Example \ref{IrJMcb}. Analogous to Theorem \ref{WDcentral}, in a wide variety of examples of modules
$M\subseteq \O_2^3$ with $\rank(M)=2$, we have checked that $M$ is Newton non-degenerate whenever
$\overline{\I_r(M)}=J_M^0$. However we have not still found a proof or a counterexample of the implication
$\textnormal{(2)}\Rightarrow\textnormal{(1)}$ of Corollary \ref{icmcentral} in general.

\item We remark that the advantage of Corollary \ref{icmcentral} over Theorem \ref{WDcentral} is that it is usually easy to verify if a module is Newton non-degenerate via Theorem \ref{resultatJLMS}. Moreover, $C^0(M)$  admits a faster computation than $C(M)$ as we can use convex-geometric methods to compute the integral closure of monomial ideals.
\end{enumerate}\end{rem}

In the following example we use Corollary \ref{icmcentral} to compute the integral closure of a family of modules.

\begin{ex}\label{exCzero}
Let us consider the submodule $M\subseteq \O_2^3$ generated by the columns of the following matrix:
$$
\left[\begin{matrix}
x^a & xy & y^a \\
y^a & x^a & xy \\
x^a+y^a & xy+x^a & y^a+xy
\end{matrix}\right],
$$
where $a\in \Z_{\geq 2}$. We remark that $\rank(M)=2$. Let $J=( x^a, xy, y^a)$. The ideal $J$ is integrally closed and
$M_1^0=M_2^0=M_3^0=J$.
An elementary computation shows that $\I_2(M)=(xy^{a+1}-x^{2a}, x^{a+1}y-y^{2a}, x^2y^2)$ and that
$\I_2(M)$ is Newton non-degenerate. Moreover $J_M^0=\overline{( x^{2a}, x^2y^2, y^{2a})}$
and then $\overline{\I_2(M)}=J_M^0$, since $\Gamma_+(\I_2(M))=\Gamma_+(J_M^0)$.
Therefore, by Corollary \ref{icmcentral}, we conclude that $\overline M=C^0(M)=C(M)$.
Given any element $h=(h_1,h_2,h_3)\in \O_2^3$, we have that $\rank(M)=\rank(M+\O_2  h)$ if and only if $h_3=h_1+h_2$ (Lemma \ref{syzygies}). Therefore
$$
\overline M=\left\{ h\in M_1^0\oplus M_2^0 \oplus M_3^0: \rank(M)=\rank(M+\O_2 h) \right\}=\left\{ [h_1\hspace{0.4cm}h_2\hspace{0.4cm}h_1+h_2]^T\in \O_2^3: h_1,h_2\in J \right\}.
$$
Therefore, a minimal generating system of $\overline M$ is given by the columns of the following matrix
$$
\left[\begin{matrix}
x^a & xy & y^a & 0 & 0 & 0 \\
0 & 0 & 0 & x^a & xy & y^a\\
x^a & xy & y^a & x^a & xy & y^a
\end{matrix}\right].
$$
\end{ex}

\begin{ex}\label{IrJMcb}
Let $M$ be the submodule of $\O_2^2$ generated by the columns of the following matrix
$$
\left[\begin{matrix}
x^3 & x^2y \\
x(x+y) & y(x+y)
\end{matrix}\right].
$$
We observe that $\rank(M)=1$. The ideal $\I_1(M)$ is given by
$$
\I_1(M)=( x^3, x^2y, x(x+y), y(x+y))=M_1+M_2=J_M.
$$
We have $\Gamma_+(\I_1(M))=\Gamma_+(x^2, y^2)$. Let $\Delta$ denote the unique compact face of dimension $1$ of $\Gamma_+(x^2, y^2)$.
Hence $(x^3)_\Delta=0$, $(x^2y)_\Delta=0$, $(x(x+y))_\Delta=x(x+y)$ and $(y(x+y))_\Delta=y(x+y)$.
Since the line of equation $y=-x$ is contained in the set of solutions of the system
$x(x+y)=y(x+y)=0$, we conclude that $I_1(M)$ is Newton degenerate.
Therefore $\overline{\I_1(M)}\neq J^0_M$ (otherwise $I_1(M)$ would be a reduction of the monomial ideal $J_M^0$ and hence
$I_1(M)$ would be Newton non-degenerate, which
is not the case). Let us observe that $\overline{\I_1(M)}=(x(x+y), y(x+y))+\m_2^3$. Therefore, by Corollary \ref{IrMh} and applying
Singular \cite{Singular} (see Remark \ref{sobreSing1}), we deduce that $\overline M=M$.

By computing explicitly a generating system of $C^0(M)=Z(M)\cap (M_1^0\oplus M_2^0)$, we also obtain that $C^0(M)=M$ and hence
$C^0(M)=\overline M$. Then
$\textnormal{(3)}\nRightarrow\textnormal{(2)}$ in Corollary \ref{icmcentral}.
\end{ex}

\begin{rem}\label{intClRemark}
Let $R$ be a Noetherian normal domain. We note that the only general approach to compute the integral closure of an arbitrary submodule $M\subseteq R^p$  is to  compute the normalization $\overline{\R(M)}$ of the Rees algebra $\R(M)$. Indeed, by  \cite{Rees87} we have $[\overline{\R(M)}]_1=\overline{M}$ and this algebra can be computed via algorithms such as the one in \cite{DeJong}, which is implemented in Macaulay2 under the command {\tt integralClosure}.

We note that  Theorem \ref{WDcentral} and Corollary \ref{icmcentral} 
can be used to compute the effectively  the  integral closure of integrally decomposable modules. Other algorithms that compute integral closures of modules under special conditions can be found in the literature (see for instance \cite[9.23]{V}).
\end{rem}

The following two examples are motivated by Example 5.8 of \cite{Kod}.

\begin{ex}\label{CMnotID}
Let us consider the submodule $M\subseteq \O_2^3$ generated by the columns of the following matrix
$$
\left[\begin{matrix}
x^2 & y   & 0 \\
0   & x   & y^2 \\
x^2 & x+y & y^2 \\
\end{matrix}\right].
$$

The rank of $M$ is $2$ and $\I_2(M)=(x^3, x^2y^2, y^3)$. Thus $\overline{\I_2(M)}=\m_2^3$. By Corollary \ref{IrMh}, we have
$\overline M=Z(M)\cap A(M)$, where
\begin{equation}\label{MbarraZM}
A(M)=\left\{h=[h_1\hspace{0.3cm}h_2\hspace{0.3cm}h_3]^T\in\O_2^3: \I_2(M,h)\subseteq \m_2^3  \right\}.
\end{equation}
In general, the submodule $Z(M)$ can be computed by using Singular \cite{Singular}, as explained in Remark \ref{sobreSing1}. In this case it is immediate to see that
$$
Z(M)=\left[\begin{matrix}
1 & 0   \\
0 & 1  \\
1 & 1 \\
\end{matrix}\right].
$$
In (\ref{MbarraZM}) the minors of size $2$ of the matrix $(M,h)$ are
$x^2h_2, yh_2-xh_1, y^2h_1, x^2(h_3-h_1), yh_3-(x+y)h_1, xh_3-(x+y)h_2$ and $y^2(h_3-h_2)$. Then
$A(M)$ is equal to the intersection of the following submodules of $\O_2^3$:
\begin{align*}
N_1&=\left\{h=[h_1\hspace{0.3cm}h_2\hspace{0.3cm}h_3]^T\in\O_2^3:  x^2h_2        \in \m_2^3 \right\}\\
N_2&=\left\{h=[h_1\hspace{0.3cm}h_2\hspace{0.3cm}h_3]^T\in\O_2^3:  yh_2-xh_1     \in \m_2^3 \right\}\\
N_3&=\left\{h=[h_1\hspace{0.3cm}h_2\hspace{0.3cm}h_3]^T\in\O_2^3:  y^2h_1        \in \m_2^3 \right\}\\
N_4&=\left\{h=[h_1\hspace{0.3cm}h_2\hspace{0.3cm}h_3]^T\in\O_2^3:  x^2(h_3-h_1)  \in \m_2^3 \right\}\\
N_5&=\left\{h=[h_1\hspace{0.3cm}h_2\hspace{0.3cm}h_3]^T\in\O_2^3:  yh_3-(x+y)h_1 \in \m_2^3 \right\}\\
N_6&=\left\{h=[h_1\hspace{0.3cm}h_2\hspace{0.3cm}h_3]^T\in\O_2^3:  xh_3-(x+y)h_2 \in \m_2^3 \right\}\\
N_7&=\left\{h=[h_1\hspace{0.3cm}h_2\hspace{0.3cm}h_3]^T\in\O_2^3:  y^2(h_3-h_2)  \in \m_2^3 \right\}.
\end{align*}
Each of the above submodules can be computed with Singular. For instance, to obtain a generating system of $N_5$ we can
use the following procedure. Let $S$ denote the quotient ring $\O_2/\m_2^3$ and let us consider the submodule
of $S^3$ given by $\syz_S(-x-y, 0, y)=\{(g_1,g_2, g_3)\in S^3: (-x-y)g_1+yg_3=0\}$. Once we have obtained a matrix of generators of
$\syz_S(-x-y, 0, y)$ with Singular, if $B$ is any submodule of $\O_2^3$ whose image in $S^3$ generates $\syz_S(-x-y, 0, y)$, then
$N_5=B+ (\m_2^3\oplus \m_2^3\oplus \m_2^3)$. Therefore it follows that $N_5$ is generated by the columns of the matrix
$$
\left[\begin{matrix}
y^2 & y^2 & xy-y^2 & x^2-xy+y^2 & y   & 0 \\
0   & 0   & y^2    & 0          & 0   & 1    \\
0   & y^2 & -y^2   & y^2        & x+y & 0
\end{matrix}\right].
$$

By computing a minimal generating system of $Z(M)\cap N_1\cap \cdots \cap N_7$, it follows that
$$
\left[\begin{matrix}
x^2 & xy & y^2 & y   & 0 \\
0   & 0  & 0   & x   & y^2  \\
x^2 & xy & y^2  & x+y & y^2 \\
\end{matrix}\right].
$$
We remark that $\overline M_1=M_1$, $\overline M_2=M_2$ and $\overline M_3=(x+y)+\m_2^2$.
Therefore, a computation with Singular shows that the module
$C(M)$, which is defined as $Z(M)\cap (\overline{M_1}\oplus\overline{M_2}\oplus \overline{M_3})$, is equal to $\overline M$.

However we have the strict inclusion $\overline{\I_2(M)}\subseteq \overline{J_M}$ in this case, since $J_M=\m_2^2$.
Hence we have $\textnormal{(3)}\nRightarrow\textnormal{(2)}$ in Theorem \ref{WDcentral}.

The inequality $\overline{\I_2(M)}\neq \overline{J_M}$ implies that $M$ is not integrally decomposable, by Theorem \ref{WDcentral}. Actually, none of the submodules $\overline{M_{\{1,2\}}}$, $\overline{M_{\{1,3\}}}$ and
$\overline{M_{\{2,3\}}}$ are integrally decomposable, by Proposition \ref{idcodfinita}, since
they $\delta(M_{1,2})=\delta(M_{1,3})=\delta(M_{2,3})=5$ and $e(M_{1,2})=e(M_{1,3})=e(M_{2,3})=8$.
\end{ex}

\begin{ex}
Let us consider the submodule $M\subseteq \O_2^2$ generated by the columns of the following matrix:
$$
\left[\begin{matrix}
x^a & y^b & 0 \\
0 & x^c & y^d \\
\end{matrix}\right],
$$
where $a,b,c,d\in\Z_{\geq 1}$. Let $I=\I_2(M)=(x^{a+c}, x^ay^d, y^{b+d})$. Since the ideals $M_1$ and $M_2$ are generated by monomials
we have, from Theorem \ref{resultatJLMS}, that
\begin{align}
\textnormal{$\overline M$ is decomposable} &\Longleftrightarrow \overline M=M_1^0\oplus M_2^0\nonumber\\
&\Longleftrightarrow\textnormal{$\overline M$ is Newton non-degenerate}\nonumber\\
&\Longleftrightarrow\textnormal{$I$ is Newton non-degenerate and $\Gamma_+(I)=\Gamma_+(M_1M_2)$}.\label{ulteq}
\end{align}
Therefore $\overline M$ is not decomposable if and only if $\Gamma_+(I)$ is strictly contained in $\Gamma_+(M_1M_2)$.
We see that $\Gamma_+(M_1M_2)=\Gamma_+(x^{a+c}, x^ay^d, y^{b+d},x^cy^b)$.
Let us observe that $\Gamma_+(I)=\Gamma_+(x^{a+c}, y^{b+d})$ if and only if $ad\geq bc$.

Let us suppose first that $ad\geq bc$. Then
$\Gamma_+(I)$ is strictly contained in $\Gamma_+(M_1M_2)$ if and only if $(c,b)$ lies below the line determined
by the two vertices of $\Gamma_+(I)$, which is to say that $ad>bc$.

If $ad<bc$, then the Newton boundary of $\Gamma_+(I)$ is equal to the union of two segments and $(c,b)$ belongs to the interior
of $\Gamma_+(I)$. Hence $\Gamma_+(I)=\Gamma_+(M_1M_2)$ and this implies that $\overline M$ is decomposable by (\ref{ulteq}).

Thus we have shown that $\overline M$ is not decomposable if and only if $ad>bc$. In this case, we have $\Gamma_+(I)=
\Gamma_+(x^{a+c}, y^{b+d})$. Let $\w=(b+d, a+c)$. By Corollary \ref{IrMh} we obtain that
\begin{align}
\overline M=\big\{h=[h_1\hspace{0.4cm}h_2]^T\in\O_2^2:\hspace{0.2cm} &d_\w(x^ah_2)\geq (a+c)(b+d),\label{embar}\\
& d_\w(y^bh_2-x^ch_1)\geq (a+c)(b+d) \hspace{0.5cm}\textnormal{and}\nonumber\\
&d_\w(y^dh_1)\geq (a+c)(b+d)\big\}.\nonumber
\end{align}
Once positive integer values are assigned to $a,b,c,d$, it is possible to obtain a generating system of $\overline M$ with Singular \cite{Singular}
by following an analogous procedure as in Example \ref{CMnotID}.
\end{ex}

In the following example we show an example of a non-decomposable integrally closed submodule of $\O_2^2$ of rank $2$ whose ideal
of maximal minors is not simple (that is, it is factorized as the product of two proper integrally closed ideals).  

\begin{ex}\label{deKod}
Let $M$ be the submodule of $\O_2^2$ generated by the columns of the following matrix:
$$
\left[\begin{matrix}
x^5 & xy   & y^5 \\
y^2 & x+y   & y^2 \\
\end{matrix}\right].
$$

We observe that $e(M_1)=10$, $e(M_2)=2$ and $e(M_1, M_2)=2$. Therefore $\delta(M)=14$. However $e(M)=22$.
Then $\overline M$ is not decomposable by Proposition \ref{idcodfinita}. In particular, since $\overline M$ is a submodule of
$\O_2^2$, $\overline M$ does not split as the direct sum of two proper ideals of $\O_2$. Let $I=\I_2(M)=( -xy^3+xy^5+y^6, -x^5y^2+y^7, -xy^3+x^6+x^5y)$.
By Corollary \ref{IrMh} it follows that
$$
\overline M=\left\{h\in\O_2^2: \I_2(M+\O_2 h)\subseteq \overline{\I_2(M)}\right\}.
$$

An easy computation shows that $I$ is Newton non-degenerate and $\Gamma_+(I)=\Gamma_+(x^6, xy^3, y^6)$.
The ideal generated by all monomials $x^{k_1}y^{k_2}$ such that $(k_1,k_2)\in\Gamma_+(I)$ is
$J=(x^6, x^5y, x^3y^2, xy^3, y^6)$. Hence $\overline I=J$ and this implies that
$$
\overline M=\left\{h=[h_1\hspace{0.4cm}h_2]^T\in\O_2^2: x^5h_2-y^2h_1,\, y^5h_2-y^2h_1,\,\, xyh_2-(x+y)h_1\in J\right\}.
$$
Hence $\overline M=N_1\cap N_2 \cap N_3$, where
\begin{align*}
N_1&=\left\{h=[h_1\hspace{0.3cm}h_2]^T\in\O_2^3:  x^5h_2-y^2h_1       \in J \right\}\\
N_2&=\left\{h=[h_1\hspace{0.3cm}h_2]^T\in\O_2^3:  y^5h_2-y^2h_1    \in J \right\}\\
N_3&=\left\{h=[h_1\hspace{0.3cm}h_2]^T\in\O_2^3:  xyh_2-(x+y)h_1          \in J \right\}.
\end{align*}

As in Example \ref{CMnotID}, using Singular \cite{Singular} we obtain that
\begin{align*}
N_1&=\left[\begin{matrix}
y^4 & x^3 & xy & 0 & 0 \\
0 & 0 & 0 & y & x
\end{matrix}\right] \hspace{2cm}
N_2
=\left[\begin{matrix}
y^4 & x^3 & xy & y^3  \\
0 & 0 & 0 & 1
\end{matrix}\right] \\
N_3&=
\left[\begin{matrix}
y^5 & x^4y-x^3y^2+x^2y^3-xy^4 & x^5-x^4y+x^3y^2-x^2y^3+xy^4 &x^2y^2-xy^3 & 0 & xy \\
0 & 0 & 0 & 0 & y^2 & x+y \\
\end{matrix}\right].
\end{align*}
Using Singular again, we have $N_3\subseteq N_1\cap N_2$. Therefore $\overline M=N_3$.
As we have discussed before, $\overline M$ is not decomposable and obviously it is integrally closed. However we have that
$$
\I_2(\overline M)=\overline{\I_2(M)}=(x^6, x^5y, x^3y^2, xy^3, y^6)=(x, y^3)(x^5, x^4y, x^2y^2, y^3).
$$
That is, the ideal $\I_2(\overline M)$ is not simple.
We also refer to \cite{Hayasaka} for another type and wide class of examples of integrally closed
and non-decomposable submodules
$N\subseteq \O_2^2$ of rank $2$ for which the ideal $\I_2(N)$ is not simple (these examples are motivated 
by a question raised by Kodiyalam in \cite[p.\,3572]{Kod}
about the converse of Theorem 5.7 of \cite{Kod}).
\end{ex}












\section*{ Acknowledgements} This work started during the stay of the authors at the
Mathematisches Forschungsinstitut Oberwolfach in June 2018.
The authors wish to thank this institution for hospitality and financial support.
The first author also acknowledges the Department of Mathematical Sciences of New Mexico State University
(Las Cruces, NM, USA) also for hospitality and financial support. We also acknowledge Prof. F. Hayasaka for informing us about
the existence of his preprint \cite{Hayasaka}. 
 The authors would like to thank the referees for their helpful  suggestions and comments.


\end{document}